\documentclass[final]{siamltex}

\usepackage{amssymb}
\usepackage{amsmath}
\usepackage{graphicx}
\usepackage{bm,bbm}
\usepackage{comment}
\usepackage{rotating}
\usepackage{color}
\usepackage{multirow}
\usepackage{multicol}
\usepackage{numprint}
\usepackage{overpic}

\usepackage{mathrsfs}

\usepackage[numbers,sort]{natbib}

\newtheorem{remark}{Remark}[section]

\newcommand{\EOD}{\end{document}}

\newcommand{\fs}{f^s}

\newcommand{\nv}{\mathbf{n}}
\newcommand{\xv}{\mathbf{x}}
\newcommand{\vv}{\mathbf{v}}
\newcommand{\Ev}{\mathbf{E}}

\newcommand{\C}   {C}
\newcommand{\Cs}  {\C^s}

\newcommand{\us}{u^s}

\newcommand{\dvv}{d\mathbf{v}}

\newcommand{\ABS}[1]{|#1|}
\newcommand{\Abs}[1]{\big|#1\big|}

\newcommand{\REAL}{\mathbbm{R}}

\newcommand{\vphi}{\varphi}

\newcommand{\vs}{v}

\newcommand{\Es}{E}
\newcommand{\Esk}{E_{k}}
\newcommand{\Esh}{\widehat{E}}

\newcommand{\ms}{m^s}
\newcommand{\qs}{q^s}

\newcommand{\dx}{dx}
\newcommand{\dv}{dv}
\newcommand{\dt}{dt}

\newcommand{\NN}{N}

\newcommand{\abs}[1]{\left|#1\right|}

\newcommand{\norm}[1]{|\!|#1|\!|}
\newcommand{\Norm}[1]{\left|\!\left|#1\right|\!\right|}

\def\trait #1 #2 #3 {\vrule width #1pt height #2pt depth #3pt}
\def\fin{\hfill
        \trait .3 5 0
        \trait 5 .3 0
        \kern-5pt
        \trait 5 5 -4.7
        \trait 0.3 5 0
\medskip}
\newcommand{\ENDPROOF}{\fin}

\newcommand{\NL}{N_L}

\newcommand{\Csnk}  {\Cs_{n,k}}
\newcommand{\Csnkp} {\Cs_{n',k'}}
\newcommand{\Csnkpp}{\Cs_{n'',k''}}

\newcommand{\Zp}{\mathbbm{Z}^{+}}
\newcommand{\ZZ}{\mathbbm{Z}}

\newcommand{\LL}{\Lambda}
\newcommand{\LS}{\Lambda_S}
\newcommand{\LG}{\Lambda_L}
\newcommand{\LH}{\Lambda_H}
\newcommand{\LF}{\Lambda_F}
\newcommand{\LN}{\Lambda^N}

\newcommand{\LSN}{\Lambda_S^N}

\newcommand{\LFN}{\Lambda_F^N}

\newcommand{\PN} {\mathcal{P}^N\!}
\newcommand{\PFN}{\mathcal{P}^{N}_{F}\!}
\newcommand{\PSN}{\mathcal{P}^{N}_{S}\!}

\newcommand{\ENs}{\Es^N}

\newcommand{\NH}{N_H}
\newcommand{\NF}{N_F}
\newcommand{\NS}{N_S}

\newcommand{\LNs}{\mathcal{L}^N}

\newcommand{\HNs}{\mathcal{H}^N}

\newcommand{\FNs}{\mathcal{F}^N}

\newcommand{\XNs}{\mathcal{X}^N}

\newcommand{\YNs}{\mathcal{Y}^N}

\newcommand{\Ox}{\Omega_x}
\newcommand{\Ov}{\Omega_\vs}

\newcommand{\nk}  {(n,k)}
\newcommand{\nkp} {(n',k')}
\newcommand{\nkpp}{(n'',k'')}

\newcommand{\CNsnk}  {\Cs^N_{n,k}}
\newcommand{\CNsnkp} {\Cs^N_{n',k'}}
\newcommand{\CNsnkpp}{\Cs^N_{n'',k''}}

\newcommand{\fNs}{f^{s,N}}
\newcommand{\fsz}{\fs_0}

\newcommand{\Fv} {\mathbf{F}}
\newcommand{\FNv}{\mathbf{F}^{N}}

\newcounter{numbs}
\newcounter{numbi}
\newcounter{numbii}

\newcommand{\SPAN}{\textsf{span}}

\newcommand{\rhoN}{\rho^N}

\renewcommand{\fs} {f}
\renewcommand{\fNs}{f^N}
\renewcommand{\Cs}{\C}

\renewcommand{\us}{u}
\renewcommand{\ms}{m}
\renewcommand{\qs}{q}

\newcommand{\gs}{g}

\newcommand{\xp}{x'}
\newcommand{\vp}{v'}
\newcommand{\dxp}{\dx'}
\newcommand{\dvp}{\dv'}
\newcommand{\LFs}{\LF\backslash{\{0\}}}
\newcommand{\LFNs}{\LFN\backslash{\{0\}}}
\newcommand{\KK}{K}
\newcommand{\KN}{K^N}

\newcommand{\ENsk}{\Esk^{N}}

\newcommand{\dS}{dS}
\newcommand{\DIVxv} {\textrm{div}}
\newcommand{\GRADxv}{\nabla}

\newcommand{\SNs}{\mathcal{S}^N}

\newcommand{\RN}{R^N}

\newcommand{\vmax}{\vs_{\max}}
\newcommand{\vmin}{\vs_{\min}}

\newcommand{\email}[1]{\textsf{e-mail: }#1}

\newcommand{\gN}{\gs^N}

\title{Convergence of spectral discretizations of the Vlasov-Poisson system}

\author{
  G. Manzini                  \footnotemark[1]
  \and D. Funaro              \footnotemark[2]
  \and G.~L. Delzanno         \footnotemark[3]
}

\begin{document}

\maketitle
\renewcommand{\thefootnote}{\fnsymbol{footnote}}
\footnotetext[1]{
  T-5 Applied Mathematics and Plasma Physics Group, 
  Los Alamos National Laboratory, Los Alamos, New Mexico, USA
  \email{gmanzini@lanl.gov}
}
\footnotetext[3]{
  Dipartimento di Scienze Fisiche, Informatiche e Matematiche,
  Universit\`a degli Studi di Modena e Reggio-Emilia, 
  Modena, Italy
  \email{daniele.funaro@unimore.it}
}
\footnotetext[3]{ 
  T-5 Applied Mathematics and Plasma Physics Group, 
  Los Alamos National Laboratory, Los Alamos, New Mexico, USA
  \email{delzanno@lanl.gov}
}
\renewcommand{\thefootnote}{\arabic{footnote}}

\begin{abstract}
  We prove the convergence of a spectral discretization of the
  Vlasov-Poisson system.
  The velocity term of the Vlasov equation is discretized using either
  Hermite functions on the infinite domain or Legendre polynomials on
  a bounded domain.
  The spatial term of the Vlasov and Poisson equations is discretized
  using periodic Fourier expansions.
  Boundary conditions are treated in weak form through a penalty type
  term, that can be applied also in the Hermite case.
  As a matter of fact, stability properties of the approximated scheme
  descend from this added term.
  The convergence analysis is carried out in details for the $1D$-$1V$
  case, but results can be generalized to multidimensional domains,
  obtained as Cartesian product, in both space and velocity.
  The error estimates show the spectral convergence, under suitable
  regularity assumptions on the exact solution.
\end{abstract}




\raggedbottom

\section{Introduction}
\label{sec:introduction}

The Vlasov-Maxwell equations, or their electrostatic equivalent
Vlasov-Poisson, describe the microscopic dynamics of a collisionless,
magnetized plasma combined with Maxwell's equation for the
electromagnetic field \cite{goldston,fitzpatrick}.
These equations are strongly coupled: the plasma provides the sources
(density and currents) for the Maxwell equations, while the
electromagnetic field moves the plasma particles via the Lorentz
force.
These equations have an intrinsic complexity, due to the fact
that they are defined in a space of six dimensions.
They are also extremely multiscale: plasma phenomena span a multitude
of spatial and temporal scales, with several orders of magnitude of
scale separation between microscopic and system scales.
Indeed, the development of methods that can describe the large-scale
dynamics of magnetized plasmas while retaining the necessary
microscopic physics is the holy grail of computational plasma physics.

\par\smallskip

There are three major classes of numerical methods for the solution of
the Vlasov-Maxwell equations, which differ by how the plasma
distribution function (i.e. phase-space density) is treated.
In the Particle-In-Cell (PIC) technique the plasma is described by
macroparticles that move through a computational mesh
\cite{birdsall,hockney}.
The Eulerian-Vlasov approach introduces a six-dimensional mesh in
space and velocity coordinates and defines the distribution function
on the mesh \cite{cheng&knorr,filbet}.
Finally, transform (spectral) methods expand the velocity part of the
distribution function in basis functions to obtain a system of 
differential equations for the coefficients of the expansion
\cite{armstrong}. 
These moment equations are defined in configuration space.
\par\smallskip

PIC is the method of choice in the plasma physics community because of
its relative simplicity, robustness and efficient parallelization on
modern computer architectures.
It is however a low-order method: reducing the well-known statistical
noise associated with the macroparticles can require a prohibitive
amount of computational resources.
Spectral methods, on the other hand, can be very accurate as one can
take advantage of the very high rate of convergence of the method,
in presence of regular solutions.
These techniques were popular in the early days of computational plasma physics,
where the Hermite or Fourier basis functions where used
\cite{armstrong}, but have not led to the development of a set of
widely used numerical codes for the plasma physics community.
Nevertheless, in recent years there has been a renewed interest for
spectral methods for Vlasov-based models
\cite{%
Camporeale-Delzanno-Bergen-Moulton:2015,%
Manzini-Delzanno-Vencels-Markidis:2016,%
Camporeale-Delzanno-Lapenta-Daughton:2006,%
Delzanno:2015,%
parker,loureiro,viriato%
}.
This is in part driven by the fact that with a suitable choice of the
spectral basis, the low-order moments of the expansion are related to
the typical fluid moments (density, momentum, energy, ...) of the
plasma.
Thus one can describe the plasma macroscopically with a few moments,
while the microscopic physics can be retained by adding more moments
to the expansion \cite{vencels2015spectral}.
This can be done only in some parts of the computational domain, as
necessary.
Thus, from a computational point of view, spectral methods might offer
an optimal way to perform large-scale simulations including
microscopic physics~\cite{Vencels:2016}. 

\par\smallskip

With this premise, this paper deals with the numerical analysis of
spectral methods for the Vlasov-Poisson equations, proving for the
first time the convergence and stability properties of the method when
the spectral basis in velocity space consists of either
Symmetrically-Weighted (SW) Hermite functions or Legendre polynomials,
combined with a Fourier discretization in space.
To this end, we consider the \emph{Vlasov-Poisson} system of equations
for the electron distribution function $\fs(\xv,\vv,t)$ (with charge
$\qs$ and mass $\ms$), and the electric field $\Ev(\xv,t)$:
\begin{align}
  \frac{\partial\fs}{\partial t} 
  + \vv\cdot\nabla_{\xv}\fs
  + \frac{\qs}{\ms}\Ev\cdot\nabla_{\vv}\fs = 0,
  \qquad
  \textrm{in~}\Omega,\,t\in[0,T[,
  \label{eq:Vlasov:3V:def}
  \\[0.5em]
  \nabla_{\xv}\cdot\Ev = \rho ~~\textrm{where}~~ \rho = n_i + \frac{\qs}{\epsilon_0}\int_{\Ov}\fs\dvv
  \qquad\textrm{in~}\Ox,\,t\in[0,T[.
  \label{eq:Poisson:3D:def}
\end{align}
Equations~\eqref{eq:Vlasov:3V:def}-\eqref{eq:Poisson:3D:def} are defined on
the \emph{six} dimensional phase space and in the time range
$t\in[0,T[$, for a given $T$.
\medskip

In equation~\eqref{eq:Poisson:3D:def} the ions are a static
neutralizing background of density $n_i$ and $\epsilon_0$ is the
vacuum electric permettivity.
We normalize these equations without loss of generality by setting
$n_i=1$, $\epsilon_0=1$, $\ms=1$, $\qs=-1$.
We assume that the phase space domain is \emph{periodic in space}, and
that the distribution function is zero at the velocity boundary or is
a \emph{rapidly decreasing function} that tends asymptotically to zero
as $e^{-\abs{\vv}^2/2}$ for $\abs{\vv}\to\pm\infty$.
To complete the mathematical formulation of the Vlasov-Poisson model
we specify the initial distribution function $\fsz(\xv,\vv)$, and we
compute the initial electric field $\Ev(\xv,0)$ by solving
equation~\eqref{eq:Poisson:3D:def} at time $t=0$.

\medskip
For exposition's sake we present the formulation and the convergence
analysis of the spectral methods in one-dimension in space and
velocity, i.e., the $1D$-$1V$ setting.
It should be clear at the end of our exposition that the type of
discretizations adopted here can be extended to multidimensional
problems with minor modifications.
We are basically thinking of situations where the variable $\xv$ is
periodic in all directions, and the variable $\vv$ is defined on a
parallelepiped.
It is a matter of redefining appropriately the way to treat boundary
conditions and use splitting properties of orthogonal projections, but
the 1D-1V derivation carries over 3D-3V lenghtly but quite
straightforwardly.

\par\smallskip

In section~\ref{sec:spectral:discretization} we introduce the spectral
discretizations using Fourier expansion for the spatial term and
either Hermite functions or Legendre polynomials for the velocity
term.
Boundary conditions are handled through a suitable penalty approach in
the velocity space.
This technique is also applied in the Hermite context.
Indeed, in alternative to the standard decay properties of Hermite
functions, zero conditions for $v$ may be enforced in weak form at the
boundaries of a bounded subset $\Ov$ of the whole space.
The integration of the Hermite spectral method on a finite sized
velocity domain is required by the convergence analysis as the
constants of the error estimates contain the size of the phase space
domain, which would blow up if the size of the velocity domain goes to
infinity.
In section~\ref{sec:stability} we reformulate the Vlasov equation and
its truncated approximation as a convection equation in a
two-dimensional phase space and we prove that both formulations are
$L^2$ stable, thanks also to the role played by the special treatment
of boundary conditions.
In section~\ref{sec:electric:field} we analyze the approximation of
the electric field and show that its error is controlled by the
approximation error of the distribution function.
In section~\ref{sec:convergence:analysis} we finally provide a
convergence analysis in the $L^2$ norm.

Further discussion and conclusions are given in
section~\ref{sec:conclusions}.

\section{Spectral discretization of the $1D$-$1V$ Vlasov-Poisson system}
\label{sec:spectral:discretization}

\subsection{The $1D$-$1V$ Vlasov-Poisson system}
\label{subsec:1D1V:formulation}

We consider the $1D$-$1V$ phase space domain $\Omega=\Ox\times\Ov$,
where the spatial subdomain is $\Ox=[0,2\pi[$ and the velocity
subdomain is either $\Ov=]\vmin,\vmax[$ or $\Ov=\REAL$.
Equations~\eqref{eq:Vlasov:3V:def}-\eqref{eq:Poisson:3D:def} become
\begin{align}
  \frac{\partial\fs}{\partial t} 
  + \vs\frac{\partial\fs}{\partial x}
  - \Es\frac{\partial\fs}{\partial\vs} = 0
  \qquad
  \textrm{in~}\Omega,\,t\in[0,T[,
  \label{eq:Vlasov:def}
  \\[0.5em]
  \frac{\partial\Es}{\partial x} = \rho ~~\textrm{where}~~ \rho = 1 -\int_{\Ov}\fs\dv
  \qquad\textrm{in~}\Ox,\,t\in[0,T[.
  \label{eq:Poisson:def}
\end{align}
Equations~\eqref{eq:Vlasov:def}-\eqref{eq:Poisson:def} are defined in
the time range $t\in[0,T[$ for some finite time $T\in\REAL$.
We assume that the phase space domain is \emph{periodic},
which implies, in particular, that $\fs(0,\vs,t)=\fs(2\pi,\vs,t)$ and
$\Es(0,t)=\Es(2\pi,t)$ for any $t\geq 0$.
When $\Ov=]\vmin,\vmax[$, we assume that
$\fs(x,\vmin,t)=\fs(x,\vmax,t)=0$ for every $x\in\Ox$ and $t\in[0,T[$,
while when $\Ov=\REAL$ we assume that $\fs$ is a \emph{rapidly
  decreasing function} in the sense that $\fs(x,\vs,t)\to 0$
asymptotically like $e^{-\vs^2/2}$ for $\vs\to\pm\infty$.
To complete the mathematical formulation of the Vlasov-Poisson model
we specify the initial distribution function $\fsz(x,\vs)$, and we
compute the initial electric field $\Es(x,0)$ by solving
equation~\eqref{eq:Poisson:def} at time $t=0$.

\subsection{Notation and technicalities}
\label{subsec:notation}
Let us introduce the following notation:

\begin{description}
  \smallskip
\item[$\bullet$] $\LL=\LS\times\LF=\Zp\times\ZZ$ is the \emph{infinite index
    range} of velocity and spatial modes. 
  Here, ``$S$'' is a generic label that may refer to either ``$H$''
  (Hermite) or ``$L$'' (Legendre), while ``$F$'' stands for
  \emph{Fourier} (periodic type).
  Throughout the paper, we will specialize when necessary the subindex
  ``$S$'' to refer specifically to the Hermite or the Legendre
  velocity representation.
  Indeed, $\LH$ or $\LG$ will be respectively the \emph{infinite
    index range} of the spectral decompositions using
  Hermite functions or Legendre polynomials;
  \smallskip
\item[$\bullet$] $\NN=(\NS,\NF)\in\mathbbm{N}^+\times\mathbbm{N}^+$ is the
  number of modes taken into account in the approximation of velocity and space,
	respectively;
  \smallskip
\item[$\bullet$] $\LN=\LSN\times\LFN=[0,\NS-1]\times[-\NF,\NF]$ is the
  \emph{finite index range} for the velocity and spatial modes of the
  \emph{truncated} distribution function.
  As a consequence, the notation $\nk\in\LN$ means $0\leq n\leq \NS-1$
  \emph{and} $-\NF\leq k\leq\NF$.
  \smallskip
\item[$\bullet$] $\vert \Omega_x\vert =2\pi$, $\vert \Omega_v\vert =\vmax-\vmin$
and $\vert \Omega\vert =\vert \Omega_x\vert \vert \Omega_v\vert $
denote the measures of the sets $\Omega_x$, $\Omega_v$ and $\Omega$,
respectively.
\end{description}

\medskip
We denote the infinite set of basis functions that are going to be
used for the velocity representation by
$\{\vphi_{n}(\vs)\}_{n\in\LS}$.
These can be either Hermite functions~\cite{Holloway:1996} or Legendre
polynomials~\cite{Manzini-Delzanno-Vencels-Markidis:2016}.
Both of them satisfy the orthogonality property:
\begin{align}
  \int_{\Ov}\vphi_{n}(\vs)\vphi_{n'}(\vs)\dv = \delta_{n,n'}.
  \label{eq:spectral:def}
\end{align}
Hermite functions are obtained as the product of Hermite polynomials
by the exponential $e^{-\vs^2/2}$.
A suitable normalization factor is then introduced in order to
have~\eqref{eq:spectral:def}.
This system is generally referred to as \emph{Symmetrically Weighted
  Hermite}, to distinguish it from the \emph{Asymmetrically Weighted
  Hermite} system. 
In the latter, two distinct sets of basis functions are considered in
a duality relationships through~\eqref{eq:spectral:def} and the
exponential $e^{-\vs^2}$ is, asymmetrically, a multiplier of 
Hermite polynomials in only one of these sets.
We shall work with the symmetric system only, although some of the
results of this paper could be adapted to cover the asymmetric case.
\par\smallskip

The situation regarding Legendre polynomials is more classical.
They are usually defined in the interval $[-1,1]$.
Via a suitable affine transformation we map them onto the interval
$[\vmin,\vmax]$, where they are successively normalized in order to
recover~\eqref{eq:spectral:def}.

\medskip
For the spatial representation we use the Fourier basis functions that
are defined and satisfy an orthogonal property as follows:
\begin{align}
  \eta_{k}(x) = \frac{1}{\sqrt{2\pi}}e^{ikx},
  \qquad\qquad
  \int_{0}^{2\pi}\eta_{k}(x)\eta_{-k'}(x)\dx = \delta_{k,k'}.
  \label{eq:Fourier:def}
\end{align}
When not strictly necessary, throughout the paper we ease the notation
by dropping out the arguments $\vs$ and $x$ from $\vphi_{n}(\vs)$ and
$\eta_{k}(x)$.
\par\smallskip

Using the orthogonal system introduced in~\eqref{eq:spectral:def}
and~\eqref{eq:Fourier:def}, we define the finite dimensional spaces
$\SNs=\SPAN\big(\{\vphi_n(\vs)\}_{n\in\LSN}\big)$;
$\FNs=\SPAN\big(\{\eta_k(x)\}_{k\in\LFN}\big)$; 
$\XNs=\SPAN\big(\{\eta_k(x)\vphi_n(\vs)\}_{k\in\LFN,n\in\LSN}\big)$.
Whenever needed, we specify the symbol $\SNs$ to $\LNs$ (Legendre) 
or $\HNs$ (Hermite). 
Afterwords, we introduce the orthogonal projection operator
$\PSN\,:\,L^2(\Omega_v)\to\SNs$ such that:
\begin{align}
  \forall\psi\in L^2(\Omega_v)\,:\,
  \int_{\Omega}(\psi-\PSN\psi )\vphi_{n}\dv = 0,
  \quad\forall n\in\LSN,
  \label{eq:proj:spectral:00}
\end{align}
and the orthogonal projection operator $\PFN\,:\,L^2(\Omega_x)\to\FNs$
such that
\begin{align}
  \forall\phi\in L^2(\Omega_x)\,:\,
  \int_{\Omega}(\phi-\PFN\phi)\eta_{k}\dx = 0,
  \quad\forall k\in\LFN.
  \label{eq:proj:Fourier:00}
\end{align}
Their extension to functions on $L^2(\Omega)=L^2(\Omega_x\times\Omega_v)$ is obvious
(it is just the matter of freezing one of the two variables), 
so that we can combine them in order to obtain the orthogonal projection
$\PN\,:\,L^2(\Omega)\to\XNs$, which comes from the composition
of operators $\PN:=\PSN\circ\PFN=\PFN\circ\PSN$.

\subsection{Spectral approximation}
\label{subsec:Galerkin:formulation}
For any $t\in[0,T[$, the Galerkin formulation for the Vlasov-Poisson system becomes:
\begin{align}
  \int_{\Omega}\vphi_{n}\eta_{k}\,\bigg(
  \frac{\partial\fs}{\partial t} 
  + \vs\frac{\partial\fs}{\partial x}
  - \Es\frac{\partial\fs}{\partial\vs}
  \bigg)\,\dv\dx & = 0
  \qquad\forall\nk\in\LL,
  \label{eq:Galerkin:formulation:A}\\[0.5em]
  \int_{\Ox}\frac{\partial\Es}{\partial x}\eta_k\,\dx =
  \int_{\Ox}\rho\eta_k\,\dx &
  \phantom{0}
  \qquad\forall k\in\LF , 
  \label{eq:Galerkin:formulation:B}
\end{align}
with the initial condition $\fs(\cdot,\cdot,0) = \fsz$. 
The function $\rho(x,t)$ in~\eqref{eq:Galerkin:formulation:B} is the
right-hand side of~\eqref{eq:Poisson:def} and represents the total
charge density of ions and electrons.
The spectral approximation of
\eqref{eq:Galerkin:formulation:A}-\eqref{eq:Galerkin:formulation:B}
reads as: \emph{Find $\fNs\in\XNs$, $\ENs\in\FNs$ such that}
\begin{align}
  \int_{\Omega}\hspace{-.1cm}\vphi_{n}\eta_{k}\,\bigg(
  \frac{\partial\fNs}{\partial t} 
  + \vs\frac{\partial\fNs}{\partial x}
  - \ENs\frac{\partial\fNs}{\partial\vs}
  \bigg)\dv\dx & = \int_{\Omega}\vphi_{n}\eta_{k}\RN\,\dv\dx
  \quad\forall\nk\in\LN,
  \label{eq:truncated:Galerkin:formulation:A}\\[0.5em]
  \int_{\Ox}\frac{\partial\ENs}{\partial x}\eta_k\,\dx =
  \int_{\Ox}\rhoN\eta_k\,\dx &
  \phantom{0}
  \qquad\forall k\in\LFN, 
  \label{eq:truncated:Galerkin:formulation:B}
\end{align}
with the initial condition $\fNs(\cdot,\cdot,0) = \PN\fsz$. Here,
$\RN$ is a kind of penalty term used to impose weakly boundary
conditions in the discrete space.
The well-posedness of
problem~\eqref{eq:truncated:Galerkin:formulation:A}-\eqref{eq:truncated:Galerkin:formulation:B},
i.e., existence and uniqueness of the numerical solutions
$\fNs\in\XNs$ and $\ENs\in\FNs$ can be proved in $]0,T]$ (for any
finite final time $T$) as discussed at the end of appendix C.
The term $\rhoN(x,t)$ in equation~\eqref{eq:truncated:Galerkin:formulation:B} is
given by:
\begin{align}
  \rhoN(x,t) = 1 - \int_{\Ov}\fNs(x,\vs,t)\dv,
\end{align}
and is an approximation of $\rho(x,t)$ in~\eqref{eq:Poisson:def}. 
Note that $\rhoN$ does not coincide with the projection $\PFN\rho$.
In the right-hand side of~\eqref{eq:truncated:Galerkin:formulation:A}, the term
$\RN$ allows us to set the boundary conditions at $\vmin$ and $\vmax$
in weak form.
This approach is similar to the penalty strategy proposed
in~\cite{Manzini-Delzanno-Vencels-Markidis:2016}.
Term $\RN$ is designed by suitably modifying the boundary term that
naturally originates from the integration by parts of the velocity
derivative on the finite domain $\Ov=]\vmin,\vmax[$.
As it will be clear in Section~\ref{sec:stability}, the special design
of $\RN$ ensures the stability of the Legendre-Fourier method.
It also guarantees that the Hermite-Fourier method, which is stable on
the infinite domain, remains stable when the integration
of~\eqref{eq:truncated:Galerkin:formulation:A} is restricted to the
finite domain $\Ov=]\vmin,\vmax[$ (see also
Remark~\ref{remark:stability:Hermite-Fourier}).
Term $\RN$ is given by the formula: 
\begin{align}
  \RN(x,\vs,t) 
  &= -\frac{1}{2}\PN\ENs(x,t)\left[
    \fNs(x,\vmax,t)\sum_{n\in\LSN}\vphi_n(\vmax)\vphi_n(\vs) 
  \right.
  \nonumber\\[0.5em]
  &\phantom{= \ENs}\quad
  \left.
    -\fNs(x,\vmin,t)\sum_{n\in\LSN}\vphi_n(\vmin)\vphi_n(\vs)
  \right] .
  \label{eq:RN:def}
\end{align}
This way of handling the boundary conditions is valid for both the
Legendre and Hermite spectral approximations and makes it possible
to develop a full stability and convergence analysis.
In the specific case of the Legendre approximation the boundary
conditions can be accounted for in several other ways.
For instance, one can impose these constraints in strong form.
Like in the so-called {\sl tau
  method}~\cite{Canuto-Hussaini-Quarteroni-Zang:2010} this can be done
by projecting the equation to be approximated in a subspace of lower
dimension ($\NS\!-\!2$ in place of $\NS$) and close the system with
two additional equations relative to the enforcement of the boundary
conditions.
An alternative is to encapsulate the boundary constraints directly in
the basis functions, but in this fashion one cannot rely on
orthogonality properties.
Both of these approaches are valid.
Nevertheless, their theoretical analysis looks harder, while we do not
expect the general performance to improve in comparison to the
approach that we are considering in this paper.
\medskip

For the Hermite discretization, on the other hand, we could choose
$\RN=0$ since the method is defined on $\Ov=\REAL$ and resorts to the
rapid decay of the Hermite functions to fulfill the zero boundary
conditions at infinity.
Nevertheless, the convergence theory for the Hermite method is
developed as for the Legendre approximation in
Section~\ref{sec:convergence:analysis} by assuming that the velocity
domain is finite and weakly imposing that $\fs$ is zero at the
velocity boundary through~\eqref{eq:RN:def}.

\par\smallskip
Throughout the paper we will also refer to
equations~\eqref{eq:truncated:Galerkin:formulation:A}
and~\eqref{eq:truncated:Galerkin:formulation:B} as the \emph{truncated
  Vlasov-Poisson system} in the equivalent form: \emph{Find
    $\fNs\in\XNs$, $\ENs\in\FNs$ such that}
\begin{align}
    \PN\left(\frac{\partial\fNs}{\partial t} 
      + \vs\frac{\partial\fNs}{\partial x}
    - \ENs\frac{\partial\fNs}{\partial\vs}
  \right) &= \RN,
  \label{eq:truncated:Vlasov:def}\\[0.5em]
  \PFN\left(\frac{\partial\ENs}{\partial x}-\rhoN\right)
  &= 0,
  \label{eq:truncated:Poisson:def}
\end{align}
where we recall that $\PN$ and $\PFN$ are the projection operators
introduced at the end of subsection~\ref{subsec:notation}.
Actually, in the first equation the action of $\PN$ can be restricted
to the second and third terms since
$\PN\big(\partial\fNs\slash{\partial t}\big)=\partial\fNs\slash{\partial
t}$. 
Also, $\PFN$ in the second equation can be removed since in the
Fourier approximation the differential operator
$\partial\slash{\partial x}$ commutes with the projector.
As this is not true in other approximation systems and in view of
possible generalizations we prefer to keep it.
The formulation~\eqref{eq:truncated:Vlasov:def}-\eqref{eq:truncated:Poisson:def}
is equivalent to a system of ordinary differential equations whose coefficients
are provided in Appendix C.

\par\medskip
\begin{lemma}
  \label{lemma:intg:gN:RN}
  Let $\RN$ be given by~\eqref{eq:RN:def}. 
  Let $\fNs$ and $\ENs$ be
  the solution of problem
  \eqref{eq:truncated:Vlasov:def}-\eqref{eq:truncated:Poisson:def}.
  Let $\gN$ be a function of $\XNs$.
  Then, it holds that:
  \begin{align}
    \int_{\Omega}\gN(x,\vs)\RN(x,\vs,t)\,\dv\dx 
    &=
    -\frac{1}{2}\int_{\Ox}\ENs(x,t)
    \Big[ 
      \fNs(x,\vmax,t)\gN(x,\vmax) 
    \nonumber\\[0.5em]
    &\hspace{1cm}
      - \fNs(x,\vmin,t)\gN(x,\vmin) 
    \Big]\,\dx.
    \label{eq:intg:gN:RN}
  \end{align}
\end{lemma}
\begin{proof}
  Since $\gN\in\XNs$, for any $x\in\Ox$ we can write:
  \begin{align}
    \gN (x,\vs) = (\PSN\gN )(x,\vs) =
    \sum_{n\in\LSN}\left(
      \vphi_n(\vs)\int_{\Ov}\gN(x,\vs')\vphi_n(\vs')\dvp
    \right).
    \label{eq:fNs:orto:decomp}
  \end{align}
  Thus, from~\eqref{eq:fNs:orto:decomp} we can derive the following
  relation:
  \begin{align}
    &\int_{\Omega}
    \gN(x,\vs)\ENs(x,t)\fNs(x,\vmax,t)\sum_{n\in\LSN}\vphi_n(\vmax)\vphi_n(\vs)\dv\dx 
    \nonumber\\
    &\quad= \int_{\Ox}
    \ENs(x,t)\fNs(x,\vmax,t)\sum_{n\in\LSN}\left(
      \vphi_n(\vmax)\int_{\Ov}\gN(x,\vs)\vphi_n(\vs)\dv
    \right)\dx 
    \nonumber\\
    &\quad= \int_{\Ox}
    \ENs(x,t)\fNs(x,\vmax,t)\gN(x,\vmax)\dx.
  \end{align}
  A similar formula for the integral of
  $\ENs(x,t)\fNs(x,\vmin,t)\gN(x,\vmin)$ is obtained with the same
  argument.
  The assertion of the lemma follows by combining these two
  relations and the definition of $\RN$ provided in~\eqref{eq:RN:def}.
\end{proof}
\par\smallskip

In particular, if we take $\gN=\fNs(\cdot,\cdot,t)$ for a given
$t\geq 0$ we find that:
\begin{align}
  \int_{\Omega}\hspace{-.2cm}\fNs(x,\vs,t)\RN(x,\vs,t)\,\dv\dx =
  -\textstyle{\frac{1}{2}}\hspace{-.1cm}\int_{\Ox}\hspace{-.2cm}\ENs(x,t)[
    \fs(x,\vmax,t)^2\hspace{-.1cm} -\hspace{-.1cm} \fs(x,\vmin,t)^2 
  ]\dx.
  \label{eq:intg:fNs:RN}
\end{align}

\begin{remark}
  Finally, it is worth mentioning that an alternative approach (not
  considered in this work) would be possible by following the dual
  Petrov-Galerkin formulation for odd-order problems proposed
  in~\cite{Shen:2003,Ma-Sun:2000}.
\end{remark}

\section{Stability}
\label{sec:stability}
We reformulate the Vlasov equation~\eqref{eq:Vlasov:def} as follows:
\begin{align}
  \frac{\partial\fs}{\partial t} + \Fv\cdot\GRADxv\fs = 0
  \qquad\textrm{where}\qquad
  \label{eq:Vlasov:bis:def} \Fv = \begin{pmatrix} \vs \\[0.25em] \displaystyle -\Es \end{pmatrix},
\end{align}
and the truncated Vlasov equation~\eqref{eq:truncated:Vlasov:def} as
follows:
\begin{align}
  \frac{\partial\fNs}{\partial t} + \PN\big(\FNv\cdot\GRADxv\fNs\big) = \RN
  \qquad\textrm{where}\qquad\FNv = \begin{pmatrix} \vs \\[0.25em] \displaystyle -\ENs \end{pmatrix},
  \label{eq:truncated:Vlasov:bis:def}
\end{align}
where $\nabla=(\partial\slash{\partial x},\partial\slash{\partial\vs})$.
The advective fields $\Fv$ and $\FNv$ are both divergence-free,
i.e., $\DIVxv\Fv=\DIVxv\FNv=0$, since $\vs$ is an independent variable
while $\Es$ and $\ENs$ do not depend on $\vs$.
This property implies that
\begin{align}
  \fs\,\Fv\cdot\GRADxv\fs 
  = \DIVxv\left(\Fv\frac{\fs^2}{2}\right) - \Big(\DIVxv\Fv\Big)\frac{\fs^2}{2}
  = \DIVxv\left(\Fv\frac{\fs^2}{2}\right).
  \label{eq:Fv:prop:1}
\end{align}
We now integrate both sides of~\eqref{eq:Fv:prop:1} on
$\Omega=\Ox\times\Ov$ and we apply the Divergence Theorem.
Furthermore, we recall that
$\fs(\cdot,\vmax,\cdot)=\fs(\cdot,\vmin,\cdot)=0$ if
$\Ov=[\vmin,\vmax]$ or that $\fs(\cdot,\vs,\cdot)\to 0$ for
$\vs\to\pm\infty$ if $\Ov=\REAL$.
We use these boundary conditions and the periodicity along the
direction $x$ to obtain:
\begin{align}
  \int_{\Omega}\fs\,\Fv\cdot\GRADxv\fs\dx\dv
  &= \int_{\Omega} \DIVxv\left(\Fv\frac{\fs^2}{2}\right)\dx\dv
  = \int_{\partial\Omega}\nv_{\partial\Omega}\cdot\Fv\frac{\fs^2}{2}\dS \nonumber\\[0.5em]
  &= \int_{\partial\Ox\times\Ov}\nv_{\partial\Omega}\cdot\Fv\frac{\fs^2}{2}\dv
  + \int_{\Ox\times\partial\Ov}\nv_{\partial\Omega}\cdot\Fv\frac{\fs^2}{2}\dx = 0,
  \label{eq:Fv:prop:2}
\end{align}
where $\nv_{\partial\Omega}$ is the outward unit vector field normal to
$\partial\Omega$.
Using~\eqref{eq:Vlasov:bis:def} and \eqref{eq:Fv:prop:2} it is easy to
arrive at a stability result in the $L^2(\Omega )$ norm for the
continuous Vlasov-Poisson system.
As a matter of fact, we have:
\begin{align}
  \frac{d}{\dt}\Norm{\fs(\cdot,\cdot,t)}_{L^2(\Omega)}^2 
  = 2\int_{\Omega}\fs\frac{\partial\fs}{\partial t}\dx\dv
  = -2\int_{\Omega}\fs\,\Fv\cdot\GRADxv\fs\dx\dv 
  = 0.
\end{align}
We formally state this result as follows.
\par\smallskip
\begin{theorem}
  Let $\fs$ be the exact solution of~\eqref{eq:Vlasov:bis:def} on the
  domain $\Omega$.
  Then, it holds that:
  \begin{align}
    \frac{d}{\dt}\Norm{\fs(\cdot,\cdot,t)}_{L^2(\Omega)}^2 = 0,
    \label{eq:Fv:prop:4}
  \end{align}
  or, equivalently, that
  $\norm{\fs(\cdot,\cdot,t)}_{L^2(\Omega)}=\norm{\fsz}_{L^2(\Omega)}$
  for every $t\in [0,T[$.
\end{theorem}

A similar result also holds regarding the truncated problem.
In this case a crucial role is played by term $\RN$.
\par\smallskip
\begin{lemma}
  \label{lemma:intg:gN:FNv}
  Let $\gN$ be a function in $\XNs$, $\FNv$ the
  advective field defined in~\eqref{eq:truncated:Vlasov:bis:def},
  $\fNs$ the solution of
  problem~\eqref{eq:truncated:Vlasov:def}-\eqref{eq:truncated:Poisson:def}
  with $\RN$ given by~\eqref{eq:RN:def}.
  Then, it holds that:
  \begin{align}
    \int_{\Omega}\gN\,\FNv\cdot\GRADxv\fNs\,\dv\dx 
    + \int_{\Omega}\fNs\,\FNv\cdot\GRADxv\gN\,\dv\dx 
    = 2\int_{\Omega}\gN\RN\dx\dv.
    \label{eq:gN:FNv:intg}
  \end{align}
\end{lemma}
\begin{proof}
  Since $\FNv$ is a zero-divergence field, using
  Lemma~\ref{lemma:intg:gN:RN} and repeating the same calculations as
  above yields:
  \begin{align*}
    &\int_{\Omega}\gN\FNv\!\!\cdot\!\GRADxv\fNs\dv\dx
    +\int_{\Omega}\fNs\FNv\!\!\cdot\!\GRADxv\gN\dv\dx
    \!=\! \int_{\Omega}\DIVxv\Big(\FNv\fNs\gN\Big)\dv\dx
    \nonumber\\[0.5em]
    &\qquad= \int_{\partial\Ox\times\Ov}\nv_{\partial\Omega}\cdot\FNv\fNs\gN\dv
    + \int_{\Ox\times\partial\Ov}\nv_{\partial\Omega}\cdot\FNv\fNs\gN\dx 
    \nonumber\\[0.5em]
    &\qquad= -\int_{\Ox}\hspace{-.2cm}\ENs(x,t)\Big(\fNs(x,\vmax,t)\gN(x,\vmax)-\fNs(x,\vmin,t)\gN(x,\vmin)\Big)\dx
     \nonumber\\[0.5em]    
     &\qquad= 2\int_{\Omega}\gN\RN\dx\dv,
  \end{align*}
  where we have used~\eqref{eq:intg:gN:RN}.
  This concludes the proof.
\end{proof}
\par\smallskip

In particular, by taking $\gN=\fNs(\cdot,\cdot,t)$
in~\eqref{eq:gN:FNv:intg} we have that:
\begin{align}
  \int_{\Omega}\fNs\,\FNv\cdot\GRADxv\fNs\,\dv\dx = \int_{\Omega}\fNs\RN\dx\dv.
  \label{eq:FNv:prop:2}
\end{align}
By putting together~\eqref{eq:truncated:Vlasov:bis:def} and~\eqref{eq:FNv:prop:2} 
we easily arrive at:
\begin{align}
  \frac{d}{\dt}\Norm{\fNs(\cdot,\cdot,t)}_{L^2(\Omega)}^2 
 & = 2\int_{\Omega}\fNs\frac{\partial\fNs}{\partial t}\dx\dv 
  = -2\int_{\Omega}\fNs\,\big(\PN\big(\FNv\cdot\GRADxv\fNs\big)-\RN\big)\dx\dv 
 \nonumber\\[0.5em]
  &= -2\int_{\Omega}\fNs\,\big(\FNv\cdot\GRADxv\fNs-\RN\big)\dx\dv
  = 0,
  \label{eq:FNv:prop:3}
\end{align}
which implies the $L^2(\Omega )$ stability of the discrete solution.
We formally state this result as follows.
\par\smallskip
\begin{theorem}
  Let $\fNs$ be the exact solution of the truncated Vlasov-Poisson
  system
  \eqref{eq:truncated:Vlasov:def}-\eqref{eq:truncated:Poisson:def}
  on the domain $\Omega$.
  Then, it holds that:
  \begin{align}
    \frac{d}{\dt}\Norm{\fNs(\cdot,\cdot,t)}_{L^2(\Omega)}^2 = 0,
    \label{eq:FNv:prop:4}
  \end{align}
  or, equivalently, that
  $\norm{\fNs(\cdot,\cdot,t)}_{L^2(\Omega)}=\norm{\PN\fsz}_{L^2(\Omega)}$
  for every $t\in [0, T[$.
\end{theorem}
\par\medskip
\begin{remark}[Stability of Hermite-Fourier method]
  \label{remark:stability:Hermite-Fourier}
  The term $\RN$ can also be used in the framework of Hermite-Fourier
  approximations.
  Actually, it is strongly suggested since it provides excellent
  stabilization properties.
  In the standard approach it is usually assumed that $\RN=0$ and the
  homogeneous boundary conditions in the infinite domain $\Ov=\REAL$
  are imposed by taking advantage of the natural decay of the Hermite
  functions.
  Nonetheless, once the Hermite-Fourier method has been set up
  according
  to~\eqref{eq:truncated:Vlasov:def}-\eqref{eq:truncated:Poisson:def},
  we can restrict the domain to $\Ov=]\vmin,\vmax[$ as in the
  Legendre-Fourier method and introduce the stabilizing term $\RN$ as
  a penalty, with the aim of enforcing zero conditions at $\vmin$ and
  $\vmax$ (in a weak sense, at least).
  Of course, this procedure is going to be effective if the size of
   $\Ov$ is large enough.
  Further comments are reported in the concluding section.
\end{remark}
\par\smallskip

\section{Approximation of the electric field}
\label{sec:electric:field}

The main result of this section is that the error on the approximation
of the electric field is controlled by that on the distribution
function.
\par\smallskip

\begin{theorem}
  \label{theo:electric:field}
  Let $\fs$, $\Es$ be the exact solution of the Vlasov-Poisson
  problem~\eqref{eq:Galerkin:formulation:A}-\eqref{eq:Galerkin:formulation:B}
  and $\fNs$, $\ENs$ be the approximations solving
  Vlasov-Poisson
  problem~\eqref{eq:truncated:Vlasov:def}-\eqref{eq:truncated:Poisson:def}.
  It holds that:
  \begin{align}
    \norm{\Es(\cdot,t)-\ENs(\cdot,t)}_{L^2(\Ox)}
    &\leq C\abs{\Ov}^{\frac{1}{2}}\norm{\fs-\fNs}_{L^2(\Omega)}.
  \end{align}
\end{theorem}
The proof of this theorem is postponed after a few technical
developments that we are going to present right away.
In the case of the Hermite-Fourier discretization, the estimate holds
by taking a finite velocity domain $\Ov=]\vmin,\vmax[$ according to
the observation in Remark~\ref{remark:stability:Hermite-Fourier}.
The path that we shall follow here allows us to generalize the
analysis to the multi-dimensional case as we mentioned in the
introduction.
Indeed, in the 1D-1V case a sharper estimate of error
$\Es(\cdot,t)-\ENs(\cdot,t)$ can be obtained by taking the difference
between~\eqref{eq:Poisson:def} and~\eqref{eq:truncated:Poisson:def}
and noting that $\PFN$ commutes with the differential operator:
\begin{align*}
  \norm{\Es(\cdot,t)-\ENs(\cdot,t)}_{H^1(\Ox)} 
  \leq C\Norm{\fs-\fNs}_{L^2(\Omega)},
\end{align*}
where $C$ is proportional to $\ABS{\Ov}^{\frac{1}{2}}$ and indipendent
of $N$.
Such a simpler approach is not allowed if we consider the divergence
operator acting on the electric field as in~\eqref{eq:Poisson:3D:def}.

\par\smallskip

The Fourier decomposition of $\Es(x,t)$ on the spatial domain $\Ox$
reads as:
\begin{align}
  \Es(x,t) = \sum_{k\in\LF}\Esk(t)\eta_{k}(x).
  \label{eq:Ek:alt:def}
\end{align}
Using~\eqref{eq:Ek:alt:def} in~~\eqref{eq:Galerkin:formulation:B} we
reformulate the $k$-th Fourier mode of the electric field as follows:
\begin{align}
  \Es_{0}(t) = 0, \qquad \Esk(t) = \frac{i}{k}\int_{\Ox\times\Ov}\fs(x,\vs,t)\eta_{-k}(x)\dx\dv
   \ \ \ \textrm{for~} k\neq 0.
  \label{eq:Fourier:coeffs:A}
\end{align}
A similar definition holds for $\ENsk$ (just substitute $\fs$ with
$\fNs$ above).
The condition that $\Es_0(t)=\ENs_0(t)=0$ is equivalent to
$\int_{\Ox}\Es(x,t)\dx=\int_{\Ox}\ENs(x,t)\dx=0$ and
has a physical motivation.
Indeed, it plays the role of a normalizing condition for the electric
field that indicates that the plasma is neutral at the macroscopic
level~\cite{Manzini-Delzanno-Vencels-Markidis:2016}.
The electric field and its approximation can be
expressed in integral form by:
\begin{align}
  \Es (x,t) &= \int_{\Omega}\fs (\xp,\vp,t)\KK(x,\xp)\dvp\dxp,\label{eq:Es:def}\\[0.5em]
  \ENs(x,t) &= \int_{\Omega}\fNs(\xp,\vp,t)\KN(x,\xp)\dvp\dxp,\label{eq:ENs:def}
\end{align}
where $\KK(x,\xp)$ is the \emph{Poisson kernel}, with the following
expression:
\begin{align}
  \KK(x,\xp) = 
  -\KK(\xp,x) =
  i\sum_{k\in\LFs}\frac{1}{k}\eta_{k}(x)\eta_{-k}(\xp).
  \label{eq:KK:def}
\end{align}
In the same way, $\KN(x,\xp)$ is the truncated version given by the
formula:
\begin{align}
  \KN(x,\xp) = 
  -\KN(\xp,x) = 
  i\sum_{k\in\LFNs}\frac{1}{k}\eta_{k}(x)\eta_{-k}(\xp).
  \label{eq:KN:def}
\end{align}
Indeed, one has
\begin{align}
  \Es(x,t) 
  &=\sum_{k\in\LF}\Esk(t)\eta_{k}(x)
  = \sum_{k\in\LFs}\bigg(\frac{i}{k}\int_{\Ox\times\Ov}\fs(\xp,\vp,t)\eta_{-k}(\xp)\dvp\dxp\bigg)\eta_{k}(x)
  \nonumber\\[0.5em]
  &= \int_{\Ox\times\Ov}\fs(\xp,\vp,t) \bigg(i\sum_{k\in\LFs}\frac{1}{k}\eta_{-k}(\xp)\eta_{k}(x)\bigg)\dvp\dxp
  \nonumber\\[0.5em]
  &= \int_{\Ox\times\Ov}\fs(\xp,\vp,t) \KK(x,\xp) \dvp\dxp.
\end{align}
A similar relation holds for $\ENs$ (just substitute $\Esk$ with
$\ENsk$ and $\fs$ with $\fNs$).

In more dimensions we may introduce the electrostatic potential $\us$
and write $\Ev=\nabla_{\xv}\us$, so that $\Delta\us=\rho$.
By expressing $\us$ as a function of $\rho$ through the Green's
function of the Dirichlet problem (typical references are for instance
\cite{Sobolev:1964,Hormander:1983}), the appropriate expression for the
kernel follows from taking the gradient of $\us$.

\par\smallskip

Moreover, both kernels $\KK(x,\xp)$ and $\KN(x,\xp)$ are real-valued
functions.
Indeed, since the complex conjugate of the Fourier basis function is
$\overline{\eta(x)}_{k}=\eta_{-k}(x)$, by swapping the summation index
from $k$ to $-k$, we note that the complex conjugate of $\KK(x,\xp)$
is given by
\begin{align*}
  \overline{\KK(x,\xp)} 
  = -i\!\!\sum_{k\in\LFs}\frac{1}{k}\eta_{-k}(x)\eta_{k}(\xp) 
  = -i\!\!\sum_{k\in\LFs}\frac{1}{-k}\eta_{k}(x)\eta_{-k}(\xp) 
  = \KK(x,\xp),
\end{align*}
and the same holds for $\KN(x,\xp)$.
These properties imply that
$\abs{\KK(x,\xp)}^2=\big(\KK(x,\xp)\big)^2$ and
$\abs{\KN(x,\xp)}^2=\big(\KN(x,\xp)\big)^2$.  
We are now ready to prove the following estimates.
\par\smallskip

\begin{lemma}
  \label{lemma:estimate:KK}
  \begin{align}
    \norm{\KK}_{L^2(\Ox\times\Ox)}^2 &= \frac{\pi^2}{3},
    \label{eq:estimate:KK} \\[0.5em]
    \norm{\KK-\KN}_{L^2(\Ox\times\Ox)}^2 &\leq 2\frac{1}{\NF}.
    \label{eq:estimate:KK-KN}
  \end{align}
\end{lemma}
\begin{proof}
  Since $\KK$ is a real-valued function, by swapping index $l$ to
  $-l$, the first inequality of the lemma is proven through the
  following development:
  \begin{align}
    &\norm{\KK}_{L^2(\Ox\times\Ox)}^2
    =\int_{\Ox\times\Ox} \big(\KK(x,\xp)\big)^2 \dx\dxp
    \nonumber\\[0.5em]
    &\qquad= -\int_{\Ox\times\Ox}\sum_{k,l\in\LFs}\frac{1}{kl}\eta_{k}(x)\eta_{-k}(\xp)\eta_{l}(x)\eta_{-l}(\xp)\dx\dxp
    \nonumber\\[0.5em]
    &\qquad= -\sum_{k,l\in\LFs}\frac{1}{k(-l)}\int_{\Ox\times\Ox}\eta_{k}(x)\eta_{-k}(\xp)\eta_{-l}(x)\eta_{l}(\xp)\dx\dxp
    \nonumber\\[0.5em]
    &\qquad=  \sum_{k,l\in\LFs}\frac{1}{kl}\,\bigg[\int_{\Ox}\eta_{k}(x)\eta_{-l}(x)\dx\bigg]\,\bigg[\int_{\Ox}\eta_{-k}(\xp)\eta_{l}(\xp)\dxp\bigg]
    \nonumber\\[0.5em]
    &\qquad=  \sum_{k,l\in\LFs}\frac{1}{kl}\,\bigg[\int_{\Ox}\eta_{k}(x)\eta_{-l}(x)\dx\bigg]\,\delta_{k,l}
    \nonumber\\[0.5em]
    &\qquad=  \sum_{k\in\LFs}\frac{1}{k^2} = 2\sum_{k=1}^{\infty}\frac{1}{k^2} = 2\frac{\pi^2}{6} = \frac{\pi^2}{3}.
  \end{align}
	
  As far as the second inequality is concerned, noting that also
    $\KK-\KN$ is a real-valued function and swapping index $l$ to $-l$, one
  has:
  \begin{align}
    &\Norm{\KK-\KN}_{L^2(\Ox\times\Ox)}^2
    =\int_{\Ox\times\Ox} \big(\KK(x,\xp)-\KN(x,\xp)\big)^2 \dx\dxp
    \nonumber\\[0.5em]
    &\qquad= \int_{\Ox\times\Ox}
    \bigg( i\!\!
    \sum_{k\not\in\LFNs} \frac{1}{k}\eta_{k}(x)\eta_{-k}(\xp)
    \bigg)^2\dx\dxp
    \nonumber\\[0.5em]
    &\qquad= -\int_{\Ox\times\Ox}\sum_{k,l\not\in\LFNs}\frac{1}{kl}\eta_{k}(x)\eta_{-k}(\xp)\eta_{l}(x)\eta_{-l}(\xp)\dx\dxp
    \nonumber\\[0.5em]
    &\qquad= -\int_{\Ox\times\Ox}\sum_{k,l\not\in\LFNs}\frac{1}{k(-l)}\eta_{k}(x)\eta_{-k}(\xp)\eta_{-l}(x)\eta_{l}(\xp)\dx\dxp
    \nonumber\\[0.5em]
    &\qquad
    = \sum_{k\not\in\LFNs}\frac{1}{k^2} 
    \leq 2(\NF)^{-1}.
  \end{align}
  The very last inequality comes from estimating the remainder of a
  convergent series.
\end{proof}

\par\smallskip
Additional estimates are reported below.
\par\smallskip

\begin{lemma}
  \label{lemma:estimate:KN}
  \begin{align}
    \abs{\KN(x,\xp)}              &\leq C\ln\NF\qquad\forall x,\xp\in\Ox ,\label{eq:KN:bound:pntwise}\\[0.5em]
    \Norm{\KN}_{L^2(\Ox\times\Ox)}^2 &\leq \frac{\pi^2}{3}( 1 + \NF^{-1}),\label{eq:KN:bound:L2}
  \end{align}
  where $C$ in the first inequality is independent of $\NF$.
\end{lemma}
\par\smallskip

\begin{proof}
  The first part is proven by noting that:
  \begin{align}
    \abs{\KN(x,\xp)} 
    = \bigg| \sum_{k\in\LFN\backslash\{0\}}\frac{1}{k}\eta_{k}(x)\eta_{-k}(\xp) \bigg|
    \leq \frac{1}{2\pi}\sum_{k\in\LFN}\left|\frac{1}{k}\right|
    \leq C\ln\NF,
    \label{eq:KN:bound}
  \end{align}
  where the final inequality follows from bounding the partial sum of
  the harmonic series.
  The second inequality follows by using~\eqref{eq:estimate:KK} and~\eqref{eq:estimate:KK-KN}:
  \begin{align}
    \Norm{\KN}_{L^2(\Ox\times\Ox)}^2 
    &\leq \Norm{\KK}_{L^2(\Ox\times\Ox)}^2 + \Norm{\KN-\KK}_{L^2(\Ox\times\Ox)}^2
    \nonumber\\[0.5em]
    &\leq \frac{\pi^2}{3} + 2\NF^{-1}
    \leq \frac{\pi^2}{3}\big( 1 + \NF^{-1}\big).
  \end{align}
\end{proof}
\par\smallskip

\noindent
We are now capable of proving theorem~\ref{theo:electric:field}.

\medskip
\noindent
\emph{Proof of Theorem~\ref{theo:electric:field}.}
Consider the approximation $\ENs$ of the electric field as
suggested in~\eqref{eq:ENs:def}. The evaluation of the $L^2(\Omega_x )$ norm
of the error requires a further integration:
\begin{align}
  &\norm{\Es(\cdot,t)-\ENs(\cdot,t)}_{L^2(\Ox)}^{2}
  \nonumber\\[0.5em]
  &=\int_{\Ox}\bigg|
  \int_{\Omega}\big(\fs(\xp,\vp,t)\KK(x,\xp)-\fNs(\xp,\vp,t)\KN(x,\xp) \big)\dxp\dvp
  \bigg|^2\dx .
  \label{eq:E:estimate:00}
\end{align}
Since $\fNs$ is orthogonal to the difference
$\KK(x,\cdot)-\KN(x,\cdot)$ it holds that 
\begin{align}
  \int_{\Omega}\fNs(\xp,\vp,t)\big(\KK(x,\xp)-\KN(x,\xp)\big)\dxp\dvp=0,
\end{align}
and we transform the inner integral of~\eqref{eq:E:estimate:00}
according to the following algebra:
\begin{align}
  &\int_{\Omega}\big(\fs(\xp,\vp,t)\KK(x,\xp)-\fNs(\xp,\vp,t)\KN(x,\xp) \big)\dxp\dvp
  \nonumber\\[0.5em]
  &=
  \int_{\Omega}
  \bigg(
  \fs(\xp,\vp,t)\big(\KK(x,\xp)-\KN(x,\xp)\big) +
  \big(\fs(\xp,\vp,t)-\fNs(\xp,\vp,t)\big)\KN(x,\xp)
  \bigg)
  \dxp\dvp
  \nonumber\\[0.5em]
  &=
  \int_{\Ox\times\Ov}
  \big(\fs(\xp,\vp,t)-\fNs(\xp,\vp,t)\big)\big(\KK(x,\xp)-\KN(x,\xp)\big)\dxp\dvp
  \nonumber\\[0.5em]
  &\phantom{=}
  +\int_{\Omega}
  \big(\fs(\xp,\vp,t)-\fNs(\xp,\vp,t)\big)\KN(x,\xp)
  \dxp\dvp.
  \label{eq:E:estimate:05}
\end{align}
Therefore, by using~\eqref{eq:E:estimate:05}
in~\eqref{eq:E:estimate:00} and the standard inequality $\vert a+b\vert^2\leq
2a^2+2b^2$, we obtain:
\begin{align}
  &\norm{\Es(\cdot,t)-\ENs(\cdot,t)}_{L^2(\Ox)}^{2}
  \nonumber\\[0.5em]
  &\leq 2\int_{\Ox}\bigg|
  \int_{\Omega}
  \big(\fs(\xp,\vp,t)-\fNs(\xp,\vp,t)\big)\big(\KK(x,\xp)-\KN(x,\xp)\big)\dxp\dvp
  \bigg|^2\dx
  \nonumber\\[0.5em]
  &\phantom{\leq}
  +2\int_{\Ox}\bigg|
  \int_{\Omega}
  \big(\fs(\xp,\vp,t)-\fNs(\xp,\vp,t)\big)\KN(x,\xp)\dxp\dvp
  \bigg|^2\dx.
  \label{eq:E:estimate:10}
\end{align}
At this point, we further bound the error with the help of the
Cauchy-Schwarz inequality:
\begin{align}
  &\norm{\Es(\cdot,t)-\ENs(\cdot,t)}_{L^2(\Ox)}^{2}
  \leq
  2\Norm{\fs-\fNs}_{L^2(\Omega)}^2\abs{\Ov}\int_{\Ox}\Norm{\KK(x,\cdot)-\KN(x,\cdot)}_{L^2(\Ox)}^2\dx
  \nonumber\\[0.5em]
  &\hspace{4.1cm}
  +2\Norm{\fs-\fNs}_{L^2(\Omega)}^2 \abs{\Ov}\int_{\Ox}\Norm{\KN(x,\cdot)}_{L^2(\Ox)}^2\dx
  \nonumber\\[0.5em]
  &\quad\leq
  2\Norm{\fs-\fNs}_{L^2(\Omega)}^2\abs{\Ov}\Big(\Norm{\KK-\KN}_{L^2(\Ox)\times L^2(\Ox)}^2
  +\Norm{\KN}_{L^2(\Ox)\times L^2(\Ox)}^2\Big)
\end{align}
(we recall that for the Hermite-Fourier method we consider the finite
domain $\Omega=]\vmin,\vmax[$).
Using the estimates of Lemma~\ref{lemma:estimate:KK}, we finally get:
\begin{align}
  \norm{\Es(\cdot,t)-\ENs(\cdot,t)}_{L^2(\Ox)}
  \leq C\abs{\Ov}^{\frac{1}{2}}\norm{\fs-\fNs}_{L^2(\Omega)},
\end{align}
where $C=2\pi(\sqrt{6}\slash{3})$.
\ENDPROOF

We end this section with a technical lemma that provides an estimate
of the $L^2(\Omega )$ norm of the electric fields $\Es$ and $\ENs$.
This result follows immediately from the estimate of the kernels $\KK$ and
$\KN$ and the stability of $\fs$ and $\fNs$. Such bounds will be used in the
convergence analysis of the next section.
\par\smallskip

\begin{lemma}
  \begin{align}
    \norm{\Es(\cdot,t)}_{L^2(\Omega)} 
    &\leq C'\,\abs{\Ov}^{\frac{1}{2}}\,\norm{\fsz}_{L^2(\Omega)},
    \label{eq:Es:bound:L2}\\[0.5em]
    \norm{\ENs(\cdot,t)}_{L^2(\Omega)} 
    &\leq C'\,\abs{\Ov}^{\frac{1}{2}}\,\big(1+\NF^{-1}\big)^{\frac{1}{2}}\norm{\PN\fsz}_{L^2(\Omega)},
    \label{eq:ENs:bound:L2}
  \end{align}
  where $C'=\sqrt{3}\pi\slash{3}$.
\end{lemma}
\par\smallskip

\begin{proof}
  In order to get inequality~\eqref{eq:Es:bound:L2} we
  use~\eqref{eq:Es:def}, the definition of $\KK$ given in
  \eqref{eq:KK:def}, and we apply the Cauchy-Schwarz inequality:
  \begin{align}
    \norm{\Es(\cdot,t)}_{L^2(\Omega)}^2 
    &= \int_{\Ox}\abs{\int_{\Omega}\fs(\xp,\vp,t)\KK(x,\xp)\dvp\dxp}^2\dx
    \nonumber\\[1.0em]
    &\leq \norm{\fs(\cdot,\cdot,t)}_{L^2(\Omega)}^2 \abs{\Ov}\int_{\Ox}\norm{\KK(x,\cdot)}_{L^2(\Ox)}^2\dx
    \nonumber\\[0.5em]
    &\leq \norm{\fsz}_{L^2(\Omega)}^2\,\abs{\Ov}\,\norm{\KK}_{L^2(\Ox\times\Ox)}^2
    \leq \norm{\fsz}_{L^2(\Omega)}^2  \abs{\Ov}\,\left(\frac{\pi^2}{3}\right).
    \nonumber
  \end{align}
  Note that
  $\norm{\fs(\cdot,\cdot,t)}_{L^2(\Omega)}=\norm{\fsz}_{L^2(\Omega)}$
  follows from the stability of $\fs$, while the estimate of the quantity
  $\norm{\KK}_{L^2(\Ox\times\Ox)}$ has been proven in
  Lemma~\eqref{lemma:estimate:KK}.
  The proof of inequality~\eqref{eq:ENs:bound:L2} follows a similar
  pattern.
  In fact, we use~\eqref{eq:ENs:def}, the definition of $\KN$ given in
  \eqref{eq:KN:def}, we apply the Cauchy-Schwarz inequality, we
  note that
  $\norm{\fNs(\cdot,\cdot,t)}_{L^2(\Omega)}=\norm{\PN\fsz}_{L^2(\Omega)}$
  from the stability of $\fNs$ and we use the estimate for
  $\norm{\KN}_{L^2(\Ox\times\Ox)}$ proven in
  Lemma~\eqref{lemma:estimate:KN}:
  \begin{align}
    \norm{\ENs(\cdot,t)}_{L^2(\Omega)}^2 
    &= \int_{\Ox}\abs{\int_{\Omega}\fNs(\xp,\vp,t)\KN(x,\xp)\dvp\dxp}^2\dx
    \nonumber\\[1.0em]
    &\leq \norm{\fNs(\cdot,\cdot,t)}_{L^2(\Omega)}^2\abs{\Ov}\int_{\Ox}\norm{\KN(x,\cdot)}_{L^2(\Ox)}^2\dx
    \nonumber\\[0.5em]
    &\leq \norm{\PN\fsz}_{L^2(\Omega)}^2\,\abs{\Ov}\,\norm{\KN}_{L^2(\Ox\times\Ox)}^2
    \nonumber\\[0.5em]
    &\leq \norm{\PN\fsz}_{L^2(\Omega)}^2\,\abs{\Ov}\,\left(\frac{\pi^2}{3}\right)\big(1+\NF^{-1}\big).
    \nonumber
  \end{align}
  This concludes the proof of the lemma.
\end{proof}
\par\smallskip

An immediate consequence of the previous lemma is an $L^2(\Omega )$
estimate of the convective fields $\Fv$ and $\FNv$, as stated by the
following result.
\par\smallskip

\begin{lemma}
  \label{lemma:Fv:FNv:bound}
  \begin{align}
    \Norm{\Fv}_{L^2(\Omega)}^2&\leq V^2\abs{\Omega}+C''\frac{\pi^2}{3},
    \label{eq:Fv:bound}\\[0.5em]
    \Norm{\FNv}_{L^2(\Omega)}^2&\leq
    V^2\abs{\Omega}+C''\frac{\pi^2}{3}(1+\NF^{-1}),
    \label{eq:FNv:bound}
  \end{align}
  with $C''=\abs{\Ov}\norm{\fsz}_{L^2(\Omega)}^2$.
\end{lemma}
\begin{proof}
  Inequality~\eqref{eq:Fv:bound} follows as a consequence
  of~\eqref{eq:Es:bound:L2} and by noting that $\abs{\vs}$ can be bounded by 
	$V:=\max \{\abs{\vmin}, \abs{\vmax}\}$.
  Likewise, we derive inequality~\eqref{eq:FNv:bound} 
  by using~\eqref{eq:ENs:bound:L2}.
\end{proof}
\par\smallskip

\section{Convergence analysis}
\label{sec:convergence:analysis}
In this section, we derive a general convergence result valid for both
Hermite-Fourier and Legendre-Fourier approximations.

\par\smallskip 

\begin{theorem}
  \label{theo:convergence:spectral}
  Let $\fs$ be the solution of the Galerkin formulation
  \eqref{eq:Galerkin:formulation:A}-\eqref{eq:Galerkin:formulation:B}
  of the Vlasov-Poisson system on the domain $\Omega=\Ox\times\Ov$
  (where $\Ov$ may be either $]\vmin,\vmax[$ or $\REAL$ ).
  Let $\fs$ belong to the Sobolev space $H^{m_F}(\Ox)\times
  H^{m_S}(\Ov)$ for some positive numbers $m_F$ and $m_S$. Moreover,
  let $\fNs$ be the solution of the truncated Vlasov-Poisson system
  \eqref{eq:truncated:Vlasov:def}-\eqref{eq:truncated:Poisson:def} on
  the domain $\Ox\times]\vmin,\vmax[$.
  Then, for any $\epsilon >0$, we have the error estimates:

  \medskip
  \begin{description}
  \item[-] for the \emph{Legendre-Fourier} method with $m_F$,
    $m_S\geq 2+\epsilon$:
    \begin{align}
      \hspace{-1cm}
      \norm{\fs(\cdot,\cdot,t)-\fNs(\cdot,\cdot,t)}_{L^2(\Omega)} 
      \leq (\NF+\NL^2)(C_1\NF^{1-m_F+\epsilon}+C_2\NL^{3/2-m_S+2\epsilon});
      \label{eq:stima:1}
    \end{align}
  \item[-] for the \emph{Hermite-Fourier} method with $m_F$, $m_S\geq 2 +\epsilon$:
    \begin{align}
      \hspace{-1cm}
      \norm{\fs(\cdot,\cdot,t)-\fNs(\cdot,\cdot,t)}_{L^2(\Omega)} 
      \leq (\NF+\sqrt{\NH})(C_1\NF^{1-m_F+\epsilon}+C_2\NH^{(1-m_S+\epsilon)/2}).
      \label{eq:stima:2}
    \end{align}
  \end{description}
  In both cases, constants $C_1$ and $C_2$ are independent of $\NN$
  ($\NN=(\NL,\NF)$ for Legendre-Fourier and $\NN=(\NH,\NF)$ for
  Hermite-Fourier), but may depend on $T$, $\vmin$ and $\vmax$.
  According to the projection estimates in Appendix B, $C_1$ and $C_2$ are
  proportional to the Sobolev norms of $\fs$.
\end{theorem}

\par\smallskip

\begin{proof}
  In view of the stability of $\fs$ and $\fNs$ and
  using~\eqref{eq:Vlasov:bis:def}
  and~\eqref{eq:truncated:Vlasov:bis:def} we find that:
  \begin{align}
    &\frac{d}{\dt}\norm{\fs-\fNs}_{L^2(\Omega)}^2 
    = \frac{d}{\dt}\left(\norm{\fs}_{L^2(\Omega)}^2 + \norm{\fNs}_{L^2(\Omega)}^2 - 2\int_{\Omega}\fs\fNs\,\dv\dx \right)
    \nonumber \\[0.5em]
    &\qquad= -2\frac{d}{\dt}\int_{\Omega}\fs\fNs\,\dv\dx
    = -2\int_{\Omega}\fNs\frac{\partial\fs}{\partial t}\dv\dx
    -2\int_{\Omega}\fs\frac{\partial\fNs}{\partial t}\dv\dx
    \nonumber \\[0.5em]
    &\qquad= 2\int_{\Omega}\fNs\big(\Fv\cdot\GRADxv\fs\big)\dv\dx
    + 2\int_{\Omega}\fs\Big(\PN\big(\FNv\cdot\GRADxv\fNs\big)-\RN\Big)\dv\dx.
    \label{eq:conv:00}
  \end{align}
  Noting that $\RN$ belongs to $\XNs$ and using the result
  of Lemma~\ref{lemma:intg:gN:FNv} with $\gN=\PN\fs(\cdot,\cdot,t)$,
  we can transform the last integral in~\eqref{eq:conv:00} as follows:
  \begin{align}
    &\int_{\Omega}\fs\big(\PN\big(\FNv\cdot\GRADxv\fNs\big)-\RN\big)\dv\dx
    = \int_{\Omega}\big(\PN\fs\big)\,\big(\FNv\cdot\GRADxv\fNs-\RN\big)\dv\dx
    \nonumber\\[0.5em]
    &\qquad
    =\frac{1}{2}\int_{\Omega}\Big( \big(\PN\fs\big)\,\FNv\cdot\GRADxv\fNs
    -\fNs\FNv\cdot\GRADxv(\PN\fs)\Big)\dv\dx .
    \label{eq:conv:05}
  \end{align}
  Now, with little algebraic manipulation we can get the identity:
  \begin{align}
    &\fNs\big(\Fv\cdot\GRADxv\fs\big) 
    +\frac{1}{2}\Big( \big(\PN\fs\big)\FNv\cdot\GRADxv\fNs - \fNs\FNv\cdot\GRADxv(\PN\fs) \Big)
    \nonumber\\[0.5em]
    &= \frac{1}{2} \Big( \big(\PN\fs-\fs\big)\FNv\cdot\GRADxv\fNs + \fNs\FNv\cdot\GRADxv\big(\fs-\PN\fs\big) \Big)
    \nonumber\\[0.5em]
    &\qquad+\big(\fNs-\fs\big)\big(\Fv-\FNv\big)\cdot\GRADxv\fs 
    +\frac{1}{2}\DIVxv\Big( \big(\Fv-\FNv\big)\fs^2 + \FNv\fNs\fs \Big) .
    \label{eq:conv:10}
  \end{align}

  We substitute~\eqref{eq:conv:05} in~\eqref{eq:conv:00} and use~\eqref{eq:conv:10}.
  Since the integral of the divergence term is zero because of the
  boundary conditions on $\fs$, we reformulate \eqref{eq:conv:00} as
  follows:
  \begin{align}
    \frac{d}{\dt}\Norm{\fs-\fNs}_{L^2(\Omega)}^2 
    = \mathcal{E}_{\textrm{proj}}(t) + \mathcal{E}_{\textrm{appr}}(t)
    \leq \abs{\mathcal{E}_{\textrm{proj}}(t)} + \abs{\mathcal{E}_{\textrm{appr}}(t)},
    \label{eq:conv:15}
  \end{align}
  where 
  \begin{align}
    \mathcal{E}_{\textrm{proj}}(t)
    &= \int_{\Omega}\Big( \big(\PN\fs-\fs\big)\FNv\cdot\GRADxv\fNs 
    + \fNs\FNv\cdot\GRADxv\big(\fs-\PN\fs\big) \Big)\dv\dx,
    \label{eq:proj:error:def}\\[0.5em]
    \mathcal{E}_{\textrm{appr}}(t)
    &= 2\int_{\Omega} \big(\fNs-\fs\big)\big(\Fv-\FNv\big)\cdot\GRADxv\fs\dv\dx.
    \label{eq:appr:error:def}
  \end{align}
    
  \medskip
  Term $\mathcal{E}_{\textrm{proj}}(t)$ depends on the
  \emph{projection} error $\PN\fs-\fs$ and its gradient; term
  $\mathcal{E}_{\textrm{appr}}(t)$ depends on the \emph{approximation}
  errors $\fNs-\fs$ and $\ENs-\Es$.
  In the next subsections, we will prove that:
  \begin{align}
    \abs{\mathcal{E}_{\textrm{proj}}(t)} &\leq \alpha(t;\NN), 
    \label{eq:Eproj:def}\\[0.5em]
    \abs{\mathcal{E}_{\textrm{appr}}(t)} &\leq \beta(t)\Norm{\fs-\fNs}_{L^2(\Omega)}^2,
    \label{eq:Eappr:def}
  \end{align}
  where $\alpha(t;\NN)\to 0$ for $\abs{\NN}\to\infty$ and $\beta(t)>0$
  is independent of $\NN$.
  The specific form of these functions depends on the choice of the
  spectral discretization and is detailed in the following subsections
  for the Legendre-Fourier method and the Hermite-Fourier method.
  Substituting~\eqref{eq:proj:error:def}
  and~\eqref{eq:appr:error:def} in~\eqref{eq:conv:15} yields
  \begin{align}
    \frac{d}{\dt}\norm{\fs-\fNs}_{L^2(\Omega)}^2 
    \leq \alpha(t;\NN) + \beta(t)\Norm{\fs-\fNs}_{L^2(\Omega)}^2.
    \label{eq:conv:20}
  \end{align}
  The assertion of the theorem follows by applying the Gronwall
  inequality and the estimates established in the next subsections.

\end{proof}

\subsection{Estimates of the projection error}

To estimate the first term of the projection
error~\eqref{eq:proj:error:def}, we first extract the supremum of
$\abs{\PN\fs-\fs}$ from the integral, we apply the Cauchy-Schwarz
inequality and we note that $\norm{\FNv}_{L^2(\Omega)}$ can be
bounded by a positive constant that is independent of $\NN$ but
depends on $\abs{\Ov}$ in view of ~\eqref{eq:FNv:bound}.
We obtain:
\begin{align}
  \left\vert \int_{\Omega}\big(\PN\fs-\fs\big)\FNv\GRADxv\fNs\,\dv\dx \right\vert
  &\leq C_1\Big(\sup_{(x,\vs)\in\Omega} \abs{\PN\fs-\fs}\Big)\,\norm{\FNv}_{L^2(\Omega)}\,
	\norm{\GRADxv\fNs}_{L^2(\Omega)}
  \nonumber\\[0.5em]
  &\leq C_2\,\norm{\fs -\PN\fs}_{H^{1+\epsilon}(\Omega)}\, \norm{\GRADxv\fNs}_{L^2(\Omega)},
  \label{eq:Eproj:1}
\end{align}
where both constants $C_1$ and $C_2$ are strictly positive and
independent of $\NN$ (they may however depend on $\fsz$ and
$\abs{\Ov}$, cf. Lemma~\ref{lemma:Fv:FNv:bound}). 
Note that $H^{1+\epsilon}(\Omega)$ with $\epsilon >0$ is included in
$L^\infty (\Omega)$, in order to justify the last inequality.

\par\medskip
To estimate the second term of the projection
error~\eqref{eq:proj:error:def} we argue in a similar way, obtaining:
\begin{align}
  \left\vert \int_{\Omega}\fNs\FNv\GRADxv\big(\PN\fs-\fs\big)\,\dv\dx\right\vert
  &\leq \Big( \sup_{(x,\vs)\in\Omega}\abs{\GRADxv(\PN\fs-\fs)}\Big)\,\norm{\FNv}_{L^2(\Omega)}\,\norm{\fNs}_{L^2(\Omega)}
  \nonumber\\[0.5em]
  &\leq C_3\norm{\fs -\PN\fs}_{H^{2+\epsilon}(\Omega)},
  \label{eq:Eproj:2}
\end{align}
where the $L^2(\Omega)$-norm of $\FNv$ is absorbed by constant $C_3$,
which is independent of $\NN$, but may still depend on $\fsz$ and
$\abs{\Ov}$.
Putting together~\eqref{eq:Eproj:1} and~\eqref{eq:Eproj:2} yields:
\begin{align}
  \abs{\mathcal{E}_{\textrm{proj}}(t) }
  \leq C_4\left(
    \norm{\GRADxv\fNs}_{L^2(\Omega)}\norm{\fs -\PN\fs}_{H^{1+\epsilon}(\Omega)}
    +\norm{\fs -\PN\fs}_{H^{2+\epsilon}(\Omega)}\right),
  \label{eq:proj:error:00}
\end{align}
where $C_4$ absorbs the previous constants and does not depend on
$\NN$.
Using standard inverse inequalities of spectral approximations, see
section~\ref{appendix:inverse-inequalities} in appendix, and recalling
that
$\norm{\fNs(\cdot,\cdot,t)}_{L^2(\Omega)}=\norm{\fNs(\cdot,\cdot,0)}_{L^2(\Omega)}=\norm{\PN\fsz}_{L^2(\Omega)}$,
$\forall t\in [0,T[$, we obtain:
\begin{align}
  \norm{\GRADxv\fNs}_{L^2(\Omega)} \leq \xi^{\NN}\norm{\PN\fsz}_{L^2(\Omega)},
  \label{eq:inverse:inequality}
\end{align}
where we introduced the auxiliary coefficient:
\begin{align}
  \xi^{\NN} =
  \begin{cases}
    \NF + \NL^2      & \textit{(Legendre-Fourier method)},\\[0.5em]
    \NF + \sqrt{\NH} & \textit{(Hermite-Fourier method)}.
  \end{cases}
\end{align}
Using~\eqref{eq:inverse:inequality}
  in~\eqref{eq:proj:error:00} yields
  \begin{align}
    \abs{\mathcal{E}_{\textrm{proj}}(t)}\leq C_{5}
    \left(
      \xi^{\NN}\norm{\fs -\PN\fs}_{H^{1+\epsilon}(\Omega)} + \norm{\fs -\PN\fs}_{H^{2+\epsilon}(\Omega)}
    \right),
  \label{eq:proj:error:10}
  \end{align}
  where the constant $C_{5}$ absorbs the $L^{2}(\Omega)$-norm of $\PN\fsz$
  and the previous constants.
\par\smallskip
  
The estimation of the bound of $ \mathcal{E}_{\textrm{proj}}$ is
concluded by applying the estimates for the projection error onto the
functional spaces $\FNs$, $\LNs$, $\HNs$ (see
section~\ref{appendix:Orthogonal-projections} in appendix).
From these estimates we may note that the error in $H^{1+\epsilon}(\Omega )$
decays faster than that in $H^{2+\epsilon}(\Omega )$, however the last one is
multiplied by $\xi^N$, so that the terms on the right-hand side
of~\eqref{eq:proj:error:10} are well balanced. 
The two estimates can be merged to obtain:
\begin{align}
  \abs{\mathcal{E}_{\textrm{proj}}(t)}\leq \xi^{\NN}\times
  \begin{cases}
    \big( C_5\NF^{1-m_F+\epsilon}+C_6\NL^{3/2-m_S +2\epsilon} \big) & \textit{(Legendre-Fourier)}, \\[0.5em]
    \big( C_5\NF^{1-m_F+\epsilon}+C_6\NH^{(1-m_S+\epsilon)/2}\big) & \textit{(Hermite-Fourier)},
  \end{cases} 
\end{align}
where the positive constants $C_5$ and $C_6$ are independent of $\NN$,
but depend on the regularity of $\fs$ through its higher-order Sobolev
norms. 
To this regard, let us note that in the Hermite-Fourier case we have
$\Omega =[0, 2\pi[\times ]\vmin , \vmax [$ in~\eqref{eq:stima:2}.
Nevertheless, the norms of $f$ on the right-hand side are evaluated in
$\Omega =[0, 2\pi[\times \REAL$.
\par\smallskip

\begin{remark}
  These estimates are, perhaps, not optimal due to the use of the
  inverse inequality in~\eqref{eq:inverse:inequality}. 
  We recall that the boundary conditions are imposed in the discrete space 
  through the penalty term $R^N$ and are only satisfied in weak form. 
  Thus, whenever one tries to modify the integrals with an integration by 
  parts, some boundary terms are produced that are difficult to estimate 
  in optimal way. 
  This is indeed the case of $\mathcal{E}_{\textrm{proj}}$ 
  in~\eqref{eq:proj:error:def} if we try to transform the gradient operator 
  of $\fNs$ into the divergence of $(\PN\fs -\fs)\FNv$, in order to avoid 
  the inverse inequality in~\eqref{eq:inverse:inequality}.
\end{remark} 

\subsection{Estimate of the approximation error}
As far as the approximation error $\mathcal{E}_{\textrm{appr}}$ is
concerned, we first note that:
\begin{align}
  \Fv(x,\vs,t)-\FNv(x,\vs,t) = \begin{pmatrix} 0\\ \displaystyle -\big(\Es(x,t)-\ENs(x,t)\big) \end{pmatrix}.
\end{align}
We plug this relation into the definition of the approximation
error~\eqref{eq:appr:error:def}.
Afterwords, we proceed with a few standard inequalities and we apply
the result of Theorem~\ref{theo:electric:field}:
\begin{align}
  \abs{\mathcal{E}_{\textrm{appr}}}
  &\leq 2 \int_{\Omega}\abs{\fNs-\fs}\,\abs{\Es-\ENs}\,\abs{\frac{\partial\fs}{\partial v}}\,\dv\dx
  \nonumber\\[0.5em]
  &\leq 2 \sup_{(x,\vs)\in\Omega}\abs{\frac{\partial\fs}{\partial\vs}}\norm{\fNs-\fs}_{L^2(\Omega)}\,\norm{\Es-\ENs}_{L^2(\Omega)}
  \leq C_7\norm{\fNs-\fs}_{L^2(\Omega)}^2,
\end{align}
where the positive constant $C_7$ is independent of $\NN$, but may
depend on the regularity of $\fs$ and $\abs{\Ov}$. 
If $C_7$ were dependent on $\NN$ we would be in trouble as $\beta (t)$
in~(\ref{eq:Eappr:def}) could grow with $\NN$.
In fact, after applying the Gronwall inequality to~(\ref{eq:conv:20}),
we get an estimate containing the term
$\exp\hspace{-.1cm}\Big(\int_0^T\beta(t)\dt\Big)$ in the right-hand
side.
This expression might become huge, thus providing a meaningless final
error estimate if $\beta$ were unbounded with respect to $\NN$.

\subsection{Further remarks on Hermite-Fourier approximations}
The Hermite-Fourier method deserves some further comments.
Although Hermite-based approximations are usually stated on
$\Ov=\REAL$, the estimate of Theorem~\ref{theo:convergence:spectral}
is derived in $L^2(\Ox\times\Ov)$ with $\Ox=[0,\,2\pi[$,
$\Ov=]\vmin,\vmax[$, and assuming that $\abs{\Ov}=\vmax-\vmin$ is
finite and hopefully not too big.
In fact, the constants of this estimate depend on the size of $\Omega$
and blow up for $\abs{\Ov}$ tending to infinity.
There are some critical issues here that we want to point it out.
First, Hermite functions are substantially different from zero on a
support that grows as $\sqrt{\NH}$.
Second, even if the exact solution $f$ has compact support, its
approximation by the Hermite functions may require a larger support as
$\NH$ grows.
Assuming that the size of $\Ov$ depends on $\NH$ may lead us to
serious drawbacks for the reasons detailed at the end of the previous
section.

\par\medskip
Note that Reference~\cite{Pulvirenti-Wick:1984}, where a similar
analysis was carried out for $\RN=0$ in the 2D-2V case, did not
address these issues.
There, estimates were given on a finite domain (only depending on time
$t$) without imposing artificial conditions and assuming (with too
much optimism, maybe) that the discretized solutions were remaining
with good approximation within the support of the exact solution
independently of $\NH$.
The above considerations teach us that domain $\Ov$ must be chosen
``wisely'' depending on the behavior manifested by the exact solution.
In particular, $\Ov$ should be large enough so that imposing zero
boundary constraints weakly is not too stringent; at the same time,
$\Ov$ must not be too large to avoid the negative influence mentioned
above on the error estimates.

\par\smallskip

\section{Conclusions}
\label{sec:conclusions}

In this paper we provided a convergence theory for the approximation
of the Vlasov-Poisson system by the symmetrically-weighted
Hermite-Fourier spectral method (restricted to a finite sized velocity
domain) and the Legendre-Fourier spectral method.

\par\smallskip

A modified weak form of the boundary conditions at the extrema of the
velocity domain made it possible to prove the stability of both
approximations.
It is well-known that the symmetrically-weighted Hermite-based
approximation is stable when the integration is on the infinite
velocity domain.
Therefore, what we proved here is that the stability remains preserved
in our formulation also when the Hermite-Fourier method integrates the
Vlasov-Poisson system on a finite velocity domain.

\par\smallskip

Finally, we note that the error estimates are weak, since they are
obtained in the $L^2(\Omega )$ norm.
For first-order nonlinear problems such as the one we are dealing
with, developing a better convergence theory could be hard.
Note that the situation in the Hermite case with $R^N=0$ (usually
employed in many applications) is even worse, since the subset
consisting of rapidly decaying functions is not closed in the
$L^2(\REAL )$ metric.
Nevertheless, this paper provides a solid theoretical foundation to
spectral methods applied to Vlasov-Poisson systems.
In addition, the penalty term $\RN$ offers a promising strategy of
handling joining conditions in multi-domain spectral approximations,
which will be the topic of further research.

\vskip3mm
\begin{center}
  \begin{large}
    {\bf Acknowledgements}
  \end{large}
\end{center}
\vskip2mm\noindent 
This work was partially funded by the Laboratory Directed Research and
Development program (LDRD), under the auspices of the National Nuclear
Security Administration of the U.S. Department of Energy by Los Alamos
National Laboratory, operated by Los Alamos National Security LLC
under contract DE-AC52-06NA25396.
The Authors gracefully thank the anonymous Reviewers for their effort and
useful suggestions.

\bibliographystyle{plain}
\bibliography{spectral}

\clearpage

\appendix

\section{Inverse inequalities}
\label{appendix:inverse-inequalities}

In this appendix and the next one, we list a series of well-known
results, limiting the exposition to the simplest case where the
indices $m$ and $r$ are integer numbers.
More general results that cover the case where $m$ and $r$ are non
integer numbers are available from the literature.
These results could be used to obtain sharper estimates but would also
increase the technicality of the exposition.

\par\smallskip

Let $\Ox=[0,2\pi[$, $\Ov=]\vmin,\vmax[$. Let $\LNs :=\SNs$ in the
Legendre case and $\HNs :=\SNs$ in the Hermite case.
Moreover, let us denote by $H^m(\Ov)$ the standard Sobolev space of
$L^2$-integrable functions whose derivatives are also $L^2$-integrable
up to order $m$.
Similarly, we have that $H_p^m(\Ox)=H_p^m(0,2\pi )$ is the
corresponding Sobolev space in the case of periodic functions.
As usual: $H^0(\Ov )=L^2(\Ov )$ and $H^0_p(0,2\pi )=L^2(0,2\pi )$.
We also consider the space $L^2_w(\REAL)$ of functions that are square
integrable with respect to the positive weight function
$w(v)=e^{-v^2}$. 
Of course, the corresponding norm is:
\begin{align}
  \norm{\psi}_{L^2_w(\REAL)} = \left(\int_{\REAL}\abs{\psi(\vs)}^2e^{-\vs^2}\dv\right)^{1/2}.
  \label{eq:proj:Hermite:10}
\end{align}
\par\smallskip
 
\noindent
Sobolev type functional spaces for a non integer $m\geq 0$ are obtained through
standard interpolation techniques.
Then, in the finite dimensional spaces, one has the following inverse
inequalities.
\par\smallskip
\begin{itemize}
\item Periodic Fourier: for all numbers $m$ and $r$ such that $0\leq
  r\leq m$ it holds that
\begin{align}
  \norm{\phi}_{H^r_p(\Ox)}\leq C\NF^{r-m}\norm{\phi}_{H^m_p(\Ox)}
  \qquad\forall\phi\in \mathcal{F}^{N_F},
\end{align}
where $C$ is independent of $\NF$;
see~\cite[Section~5.8.1]{Canuto-Hussaini-Quarteroni-Zang:2010}
\item Legendre polynomials: for all  $r\geq 1$ it holds that
\begin{align}
  \Norm{\frac{\partial^r\phi}{\partial\vs^r}}_{L^2(\Ov)}\leq C\NL^{2r}\norm{\phi}_{L^2(\Ov)}
  \qquad\forall\phi\in\mathcal{L}^{N_L},
\end{align}
where the constant $C$ is independent of $\NL$ but depends on
$\abs{\Ov}$;
see~\cite[Section~9.4.1]{Canuto-Hussaini-Quarteroni-Zang:1988}
\item Hermite polynomials:
\begin{align}
  \Norm{\frac{\partial\phi}{\partial\vs}}_{L^2_w(\REAL)}\leq C\sqrt{\NH}\norm{\phi}_{L^2_w(\REAL)},
\end{align}
for all polynomials of degree at most $\NH$, where the constant $C$ is
independent of $\NH$.
A similar result holds in the case of Hermite functions
($\phi\in\mathcal{H}^{N_H}$).
In this case the weight function is $w(v)=e^{v^2}$. 
One can pass from a case to the other by virtue
of~\cite[Lemma~6.7.4]{Funaro:1992}.

\end{itemize}
\par\smallskip

\section{Orthogonal projections}
\label{appendix:Orthogonal-projections}

\begin{itemize}
\item Periodic Fourier: consider the operator $\mathcal{P}_F^{N_F}$,
  which projects $L^2(0,2\pi)$ onto
  $\textrm{span}\{\eta_k\}_{k\in\LFN}$.
  We have the following estimate for the projection error:
  \begin{align}
    \norm{\psi-\mathcal{P}_F^{N_F}\psi}_{H^r_p(0,2\pi)}
    \leq C\NF^{r-m}\Norm{\frac{\partial^{m}\psi}{\partial x^{m}}}_{L^2(0,2\pi)},
    \label{eq:proj:Fourier:error:10}
  \end{align}
  which holds for every $0\leq r\leq m$.
  More details are found
  in~\cite[Section~5.1.2]{Canuto-Hussaini-Quarteroni-Zang:2010}.

\item Legendre polynomials: consider the operator
  $\mathcal{P}_L^{N_L}$, which projects $L^2(\Ov)$ onto the space of
  polynomials of degree at most $\NL$.
  We have the following estimate:
  \begin{align}
    \norm{\psi-\mathcal{P}_L^{N_L}\psi}_{L^2(\Ov)}
    \leq C\NL^{-m}\norm{\psi}_{H^m(\Ov)},
    \label{eq:proj:Lagrange:20}
  \end{align}
  where $m\geq 0$ and the constant $C$ is independent of $N_L$
  but depends on $\abs{\Ov}$.  
  An extension, where at the left-hand side we find higher Sobolev
  norms is available:
  \begin{align}
    \norm{\psi-\mathcal{P}_L^{N_L}\psi}_{H^r(\Ov)}
    \leq C\NL^{2r-1/2-m}\norm{\psi}_{H^m(\Ov)},
    \label{eq:proj:Lagrange:21}
  \end{align}
  where $m\geq r\geq 1$.
  More details are found
  in~\cite[Section~5.4.2]{Canuto-Hussaini-Quarteroni-Zang:2010}.

\item Hermite polynomials: 
  we have the following estimate for the projection error in the space
  of polynomials of degree at most $\NH$:
  \begin{align}
    \norm{\psi-\mathcal{P}_H^{N_H}\psi}_{L^2_w(\REAL)}
    \leq C\NH^{-m/2}\Norm{\frac{\partial^{m}\psi}{\partial\vs^{m}}}_{L^2_w(\REAL)},
    \label{eq:proj:Hermite:15}
  \end{align}
  which holds for any $\psi\in H^m_w(\REAL)$, $m\geq 0$;
  see~\cite[Theorem~6.2.6]{Funaro:1992}.  
  This can be also generalized when higher-order norms are present on
  the left-hand side (see \cite{Shen-Tang-Wang:2011}, Theorem 7.13,
  p.270):
  \begin{align}
    \norm{\psi-\mathcal{P}_H^{N_H}\psi}_{H^r_w(\REAL)}
    \leq C\NH^{(r-m)/2}\Norm{\frac{\partial^{m}\psi}{\partial\vs^{m}}}_{L^2_w(\REAL)},
    \label{eq:proj:Hermite:16}
  \end{align}
  Similar results hold in the case of Hermite functions where the
  weight is $w(v)=e^{v^2}$.
\end{itemize}
\par\smallskip

We are ready to provide an estimate to the projection operator for
functions of both variables $x$ and $v$.
We begin by observing that for $f\in L^2(\Omega )$ one has:
\begin{align}
  (I-\PN)\fs = (I-\mathcal{P}_F^{N_F})\fs 
	+\mathcal{P}_F^{N_F}(I-\mathcal{P}_S^{N_S})\fs  ,
  \label{eq:proj:global:error:00}
\end{align}
By virtue of this equality we get:
\begin{align}
  \norm{\fs-\PN\fs}_{L^2(\Omega)}
	&\leq \norm{\fs-\mathcal{P}_F^{N_F}\fs}_{L^2(\Omega)} +
  \norm{\mathcal{P}_F^{N_F}}_{\mathcal{L}\big(L^2(\Omega),L^2(\Omega)\big)}
	\norm{\fs-\mathcal{P}_S^{N_S}\fs}_{L^2(\Omega)}
  \nonumber\\
  &\leq \norm{\fs-\mathcal{P}_F^{N_F}\fs}_{L^2(\Omega)}+
	\norm{\fs-\mathcal{P}_S^{N_S}\fs}_{L^2(\Omega)}
  \label{eq:proj:global:error:05}
\end{align}
since the norm of the projector $\mathcal{P}_F^{N_F}$ in
$\mathscr{L}\big(L^2(\Omega),L^2(\Omega)\big)$ is less than one.

\par\smallskip

The bound of the last term can be specialized according to the method
adopted.
This gives the final estimates:
\begin{align}
  \norm{f -\PN\fs}_{L^2(\Omega)}
  \leq 
  \begin{cases}
    C_1 \NF^{-m_F} + C_2 \NL^{-m_S}    \quad \ \ m_F, m_S\geq 0 & \textit{(Leg.-Fou.)},\\[0.5em]
    C_1 \NF^{-m_F} + C_2 \NH^{-m_S/2}  \quad m_F, m_S\geq 0 & \textit{(Her.-Fou.)}.
  \end{cases}
\end{align}
In the Hermite case one has $\Ov=\REAL$.
As before, the constants $C_1$ and $C_2$ do not depend on the
discretization parameters.
They depend however on Sobolev type norms type norms of the given
function $f\in H^{m_F}(\Omega_x) \times H^{m_S}(\Omega_v)$.
When the Sobolev space on the left-hand side is $H^r(\Omega)$, with
$r\geq 1$, we have:
\begin{align}
  \norm{f -\PN\fs}_{H^r(\Omega)}
  \leq 
  \begin{cases}
    C_1 \NF^{r-m_F} + C_2 \NL^{2r-1/2-m_L}    \ m_F, m_L\geq r & \hspace{-.2cm} \textit{(Leg.-Fou.)},\\[0.5em]
    C_1 \NF^{r-m_F} + C_2 \NH^{(r-m_H)/2}   \quad m_F, m_H\geq r & \hspace{-.2cm}\textit{(Her.-Fou.)}.
  \end{cases}
\end{align}
A way to prove the above result is to differentiate the equality
in~\eqref{eq:proj:global:error:00} and note that the Fourier projector
commutes with derivatives.
Successively, one makes use of~\eqref{eq:proj:Lagrange:21}
and~\eqref{eq:proj:Hermite:15}.
Regarding these techniques, we refer to the original paper
\cite{Canuto-Quarteroni:1982} for more insight.

\section{Implementation}

\subsection{Spectral decomposition of the Vlasov equation}
\label{subsec:spectral:decomp:Vlasov}
The spectral decomposition of $\fs(x,\vs,t)$ on the space-velocity
domain $(x,\vs)\in\Ox\times\Ov$ and $t\in[0,T[$ reads as:
\begin{align}
  \fs(x,\vs,t) = \sum_{\nk\in\LL}\Csnk(t)\vphi_n(\vs)\eta_k(x).
  \label{eq:fs:decomposition}
\end{align}
By substituting~\eqref{eq:fs:decomposition}
in~\eqref{eq:Galerkin:formulation:A} we obtain the following infinite
non-linear system of ordinary differential equations for the
coefficients $\Csnk(t)$:
\begin{align}
  &\frac{d\Csnk}{\dt} 
  + \sum_{\nkp\in\LL}A_{\nk,\nkp}\Csnkp
  \nonumber\\ 
  &
  \phantom{\frac{d\Csnk}{\dt}}
  + \sum_{\nkp,\nkpp\in\LL}B_{\nk,\nkp,\nkpp}\Csnkp\Csnkpp = 0
  \quad\forall\nk\in\LL,
  \label{eq:ODE:exact:A}
\end{align}
with the initial conditions $\Csnk(0) = \Cs^{0}_{n,k}$, $\forall\nk\in\LL$,
which are obtained from the spectral expansion of the initial solution
$\fs_0(x,\vs)$.
This writing is not used in the theoretical analysis; it is important
however for the implementation of the algorithm, since the
coefficients $A_{\nk,\nkp}$ and $B_{\nk,\nkp,\nkpp}$ are the same as
those of the numerical approximation. We are now going to show how to
compute them.
The \emph{linear term} in~\eqref{eq:ODE:exact:A} is such that:
\begin{align}
  \sum_{\nkp\in\LL}A_{\nk,\nkp}\Csnkp
  = \int_{\Omega}\vphi_{n}\eta_{k}\,\vs\frac{\partial\fs}{\partial x}\,\dv\dx
\end{align}
with
\begin{align}
  A_{\nk,\nkp} 
  = \int_{\Omega}\vphi_{n}\eta_{k}\,\vs\frac{\partial}{\partial x}\big(\vphi_{n'}\eta_{k'}\big)\,\dv\dx.
  \label{eq:def:A}
\end{align}
Note that $A_{\nk,\nkp}$ does not depend on $t$ because such a
dependence is clearly expressed through the coefficient $\Csnkp(t)$.
The coefficient in~\eqref{eq:def:A} can be recovered using the
orthogonality properties of Hermite or Legendre polynomials, and in
particular by the respective three-term recursion
formulas~\cite{Delzanno:2015,Camporeale-Delzanno-Bergen-Moulton:2015,Manzini-Delzanno-Vencels-Markidis:2016}.
\par\smallskip

In a similar manner, the \emph{non-linear term}
in~\eqref{eq:ODE:exact:A} is such that:
\begin{align}
  \sum_{\nkp,\nkpp\in\LL}B_{\nk,\nkp,\nkpp}\Csnkp\Csnkpp
  = -\int_{\Omega}\vphi_{n}\eta_{k}\,\Es\frac{\partial\fs}{\partial v}\,\dv\dx.
  \label{eq:non-linear:term:B}
\end{align}
This time, to derive the expression of the coefficients $B_{\nk,\nkp,\nkpp}$ we
first need to write the electric field $\Es$ in terms of the coefficients
$\Csnk$ in~\eqref{eq:non-linear:term:B}.
For this purpose, we consider the decomposition:
\begin{align}
  \Es(x,t) = \sum_{\nk\in\LL}\Esh_{n,k}(x)\,\Csnk(t),
  \label{eq:def:E:fourier}
\end{align}
basically corresponding to the Fourier expansion of $\Es$ and that
will be discussed in the next subsection (in particular, $\Esh_{n,k}$
is given by formula~\eqref{eq:def:Eh:fourier}).
Combining~\eqref{eq:fs:decomposition} and~\eqref{eq:def:E:fourier},
from~\eqref{eq:non-linear:term:B} we find that:
\begin{align}
  B_{\nk,\nkp,\nkpp} = -\int_{\Omega}\vphi_{n}\eta_{k}\,\Esh_{n',k'}\frac{\partial}{\partial \vs}\big(\vphi_{n''}\eta_{k''}\big)\dv\dx.
  \label{eq:def:B}
\end{align}
As already noted for $A_{\nk,\nkp}$, also the coefficient
$B_{\nk,\nkp,\nkpp}$ does not depend on time.

\subsection{Spectral decomposition of the Poisson equation}
\label{subsec:spectral:decomp:Poisson}
Consider the Fourier decomposition of $\Es(x,t)$ on the spatial domain
$\Ox$ given by~\eqref{eq:Ek:alt:def}.
By substituting~\eqref{eq:fs:decomposition} and~\eqref{eq:Ek:alt:def}
in~\eqref{eq:Galerkin:formulation:B} we obtain:
\begin{align}
  \sum_{k\in\LF}(ik)\Es_k\eta_{k} = 
  \sqrt{2\pi}\eta_{0}
  -\sum_{\nk\in\LL}\Cs_{n,k}\eta_{k}\int_{\Ov}\vphi_{n}\dv.
  \label{eq:Ek:Fourier:00}
\end{align}
We multiply~\eqref{eq:Ek:Fourier:00} by $\eta_{-k'}$, integrate on
$\Ox$ and use the orthogonality property~\eqref{eq:Fourier:def} to
obtain (changing the summation index back to $k$):
\begin{align}
  \Es_k(t) = \sum_{n\in\LS}\gamma_{n,k}\Csnk(t)
  \quad\forall k\in\LF,
  \label{eq:Ek:def}
\end{align}
where for all $n\in\LS$ we take:
\begin{align}
  \gamma_{n,0}=0\qquad {\rm and }\qquad \gamma_{n,k} = \frac{i}{k}\int_{\Ov}\vphi_{n}(\vs)\dv
  \qquad k\in\LF\backslash\{0\}.
  \label{eq:gamma:def}
\end{align}

\noindent
From the orthogonality properties of the Symmetrically Weighted
Hermite functions it holds that $\gamma_{n,k}=0$ for every $k$ and odd
$n$.
Instead, for Legendre polynomials it holds that $\gamma_{n,k}=0$ for
every $n>0$.
From $\gamma_{n,0}=0$ for all $n$ it follows that $\Es_0(t)=0$, which
is equivalent to $\int_{\Ox}\Es(x,t)\dx=0$.

\par\smallskip

Using the definition of $\Es_{k}$ in~\eqref{eq:Ek:def}, by
comparing~\eqref{eq:Ek:alt:def} and~\eqref{eq:def:E:fourier} we
immediately find that:
\begin{align}
  \Esh_{n,k}(x) 
  = \gamma_{n,k}\eta_k(x)
  = \frac{i}{k}\bigg(\int_{\Ov}\vphi_{n}(\vs)\dv\bigg)\,\eta_{k}(x)
  \quad\,k\in\LF\backslash{\{0\}}.
  \label{eq:def:Eh:fourier}
\end{align}

\newcommand{\uvec}[1]{\underline{#1}}
\newcommand{\AN}{\mathcal{A}^{\NN}}
\newcommand{\BN}{\mathcal{B}^{\NN}}
\newcommand{\calLN}{\mathcal{L}^{\NN}}

\subsection{Global discretization}
Consider the \emph{approximated distribution function}:
\begin{align}
  \fNs(x,\vs,t) = \sum_{\nk\in\LN}\CNsnk(t)\vphi_n(\vs)\eta_k(x),
  \label{eq:fNs:def}
\end{align}
which approximates the function $\fs(x,\vs,t)$ in $\XNs$,
as well as the \emph{approximated electric field}:
\begin{align}
  \ENs(x,t) = \sum_{\nk\in\LN}\gamma_{n,k}\CNsnk(t)\eta_k(x),
  \label{eq:def:E:fourier:trunc}
\end{align}
which approximates $\Es(x,t)$ in $\FNs$.
Clearly, $\fNs$ does not coincide with projection $\PN\fs$ and $\ENs$
with projection $\PFN\Es$.
The coefficients $\CNsnk$ are determined by imposing that $\fNs$ and
$\ENs$ are the solution of the \emph{truncated Vlasov-Poisson system}
given
by~\eqref{eq:truncated:Vlasov:def}-\eqref{eq:truncated:Poisson:def}.
From~\eqref{eq:truncated:Vlasov:def}-\eqref{eq:truncated:Poisson:def},
a straightforward calculation yields that the coefficients $\CNsnk(t)$
are the solution of a system of ordinary differential equations,
namely:
\begin{align}
  &\frac{d\CNsnk}{\dt} 
  + \sum_{\nkp\in\LN}\hspace{-2mm}A_{\nk,\nkp}\CNsnkp 
  + \sum_{\nkp,\nkpp\in\LN}\hspace{-8mm}
  \Big(
    B_{\nk,\nkp,\nkpp}
  \nonumber\\[0.5em]
  &\qquad
    -\widetilde{B}_{\nk,\nkp,\nkpp}
  \Big)\CNsnkp\CNsnkpp 
  = 0
  \qquad\forall\nk\in\LN,
  \label{eq:ODE:trunc:A}
\end{align}
and at $t=0$ we set $\CNsnk(0) = \Cs^{0}_{n,k}$, $\forall\nk\in\LN$
using the same initial conditions of problem~\eqref{eq:ODE:exact:A}.
The coefficients $A$ and $B$ are the same as in \eqref{eq:def:A} and
\eqref{eq:def:B}, respectively.
Instead, the coefficients denoted by $\widetilde{B}$ are obtained from
integration of term $\RN$ and using the result of
Lemma~\ref{lemma:intg:gN:RN} to derive their explicit formula.
In fact, with the special choice 
$\gN(x,t)=\eta_{k}(x)\vphi_n(\vs)$ in~\eqref{eq:intg:gN:RN} we find that: 
\begin{align*}
  \int_{\Omega}\eta_{k}(x)\vphi_n(\vs)\RN(x,\vs,t)\dv\dx 
  &= -\frac{1}{2}\int_{\Ox}\ENs(x,t)\eta_{k}(x)
  \Big[
    \fNs(x,\vmax,t)\vphi_{n}(\vmax) 
  \nonumber\\[0.5em]
  &\hspace{1cm}    
    - \fNs(x,\vmin,t)\vphi_{n}(\vmin)
  \Big]\,\dx.
\end{align*}
Using the expansions of $\fNs(x,\vmax,t)$ and $\fNs(x,\vmin,t)$,
cf.~\eqref{eq:fs:decomposition}, and the expression of the electric
field shown in~\eqref{eq:def:E:fourier}, one gets:
\begin{align*}
  \int_{\Omega}\eta_{k}(x)\vphi_n(\vs)\RN(x,\vs,t)\dv\dx
  &= \sum_{\nkp,\nkpp}\hspace{-0.2cm}\CNsnkp(t)\CNsnkpp(t)\Big[
    \vphi_{n}(\vmax)\vphi_{n'}(\vmax) 
    \nonumber\\[0.5em] 
    &- \vphi_{n}(\vmin)\vphi_{n'}(\vmin)
  \Big]
  \int_{\Ox}\Esh_{\nkpp}(x)\eta_{k}(x)\eta_{k'}(x)\dx
  \nonumber\\[0.5em]
 & =  \sum_{\nkp,\nkpp}\hspace{-0.3cm} \CNsnkp\CNsnkpp\widetilde{B}_{\nk,\nkp,\nkpp},
\end{align*}
from which one finally recovers:
\begin{align}
  \widetilde{B}_{\nk,\nkp,\nkpp} 
  = \Big[
  \vphi_{n}(\vmax)\vphi_{n'}(\vmax) 
  - \vphi_{n}(\vmin)\vphi_{n'}(\vmin)
  \Big]
  \int_{\Ox}\Esh_{\nkpp}\eta_{k}\eta_{k'}\dx.
  \nonumber
\end{align}

Since coefficients $A_{\nk,\nkp}$, $B_{\nk,\nkp,\nkpp}$, and
$\widetilde{B}_{\nk,\nkp,\nkpp}$ do not depend on $t$, the
well-posedness of the present system of ordinary differential
equations follows from classical results.
Indeed, the forcing term is the sum of a linear and a quadratic part.
Such a term is locally Lipschitz, so that existence and uniqueness in
a suitable interval $[0,t_N[\subset[0,T[$ may be recovered through a
contraction theorem.
By the way, this result can be extended to the whole interval $[0,T[$
(and, possibly, to the entire semi-axis $[0, +\infty[$).
It is enough to recall equality~\eqref{eq:Fv:prop:4}, from which we
deduce that the quantity $\sum_{(n,k)\in\LN}\ABS{\CNsnk}^2$ is bounded
by a constant $\kappa^N$ that is independent of $t$.
This prevents the blow up of the solution in a finite time.
More in detail, we put equation~\eqref{eq:ODE:trunc:A} in vector form:
\begin{align}
  \frac{d\CNsnk}{dt} 
  + \mathcal{A}^{\NN}_{(n,k)}\uvec{C}^{\NN}
  + \mathcal{B}^{\NN}_{(n,k)}(\uvec{C}^{\NN},\uvec{C}^N) =0 
  \qquad\forall\nk\in\LN
\end{align}
where $\AN_{(n,k)}$ represents the linear part, $\BN_{(n,k)}$ the
quadratic one (including the penalty term
$\widetilde{B}_{\nk,\nkp,\nkpp}$), and vector
${\uvec{C}^{\NN}}=\{ \CNsnk \}_{\nk\in\LN}$ collects all the spectral modes.
We also define the closed subspace $\YNs$ of $\XNs$
consisting of all functions bounded in $L^2(\Omega )$ by a given
constant $\kappa^N$, independent of $t$.
For two vectors $\uvec{C}^{\NN}_1$ and $\uvec{C}^{\NN}_2$, it holds
that, for any $\nk\in\LN$, $\BN_{(n,k)}$ is Lipschitz in
$\YNs\times\YNs$ by virtue of the following inequalities:
\begin{align*}
  &\Abs{
    \BN_{(n,k)}(\uvec{C}^{\NN}_{1},\uvec{C}^{\NN}_{1})-
    \BN_{(n,k)}(\uvec{C}^{\NN}_{2},\uvec{C}^{\NN}_{2})}
  \\[0.5em]
  &\quad\leq
  \Abs{\BN_{(n,k)}(\uvec{C}^{\NN}_{1}-\uvec{C}^{\NN}_{2},\uvec{C}^{\NN}_{1})}+
  \Abs{\BN_{(n,k)}(\uvec{C}^{\NN}_{2},\uvec{C}^{\NN}_{1}-\uvec{C}^{\NN}_{2})}
  \leq
  \lambda^N\Vert\uvec{C}^{\NN}_1 -\uvec{C}^{\NN}_2\Vert_{\NN},
\end{align*}
where $\Vert\cdot\Vert_N$ denotes the classical norm in finite
dimension and $\lambda^N >0$ is the Lipschitz constant.
To obtain the above estimate, we used the Schwarz inequality and the
boundedness in $\YNs$, i.e., $\Vert\uvec{C}^{\NN}_i\Vert_N\leq\kappa^N$
for $i=1,2$.
By observing that the coefficients of $\BN_{\nk}$ do not depend on
time we obtain that the Lipschitz constant $\lambda^N$ also does not
depend on time.
We can actually fix $\lambda^N$ in such a way that it does not depend on
the indices $\nk$.
These last remarks allow for the prolongation of the discrete solution
to the additional interval $[t_N, 2t_N[$, which is of the same size of
the initial one.
Such a procedure can be repeated as many times is necessary to cover
the interval $[0, T[$ (and beyond).

\medskip
The solution of the system of ordinary differential equations is
infinitely times differentiable in space and time.
However the existence proof provided above is naturally posed in the
space $C^0(0,T;L^2(\Omega))$, which is where we prove the error
estimates.

\end{document}